\documentclass[preprint,nonatbib]{elsarticle}

\usepackage{lineno}
\makeatletter
\let\c@author\relax
\makeatother

\usepackage{geometry} 
\usepackage{comment}
\usepackage{hhline}
\usepackage{algorithm}
\usepackage{algorithmic}
\usepackage[ngerman,english]{babel}
\usepackage[T1]{fontenc}
\usepackage[utf8]{inputenc}
\usepackage[table,xcdraw]{xcolor}
\usepackage{graphicx} 
\usepackage{floatflt} 
\usepackage{amsmath}  
\usepackage{amssymb}  
\usepackage{amsthm}   
\usepackage{paralist} 
\usepackage{color}	
\usepackage{bbm} 
\usepackage{tikz} 
\usepackage{tikz-cd} 
\usepackage{array}
\usepackage{pgfplots} 
\usepackage[section]{placeins} 
\usepackage{subcaption}
\usepackage{caption}
\usepackage[official]{eurosym}
\usepackage{multirow}
\usepackage{float}
\usepackage{verbatim}
\usepgfplotslibrary{units}
\usepackage{units}
\usepackage{mathtools}
\usetikzlibrary{arrows,shapes,positioning,matrix}
\usepackage{smartdiagram}
\usepackage{booktabs}
\usepackage[b]{esvect}
\usepackage[colorlinks=false, pdfborder={0 0 0}]{hyperref}
\usepackage[nameinlink,noabbrev]{cleveref}

\usepackage{tikz-network}
\usetikzlibrary{positioning}
\usepackage[flushleft]{threeparttable}
\usepackage{pifont}
\usepackage{setspace}
\usepackage{bm}
\usepackage{mathrsfs}
\usepackage[normalem]{ulem}

\newtheoremstyle{remarkitalic}% name
  {3pt}% Space above
  {3pt}% Space below
  {\itshape}% Body font      <-- normally upright; we change it to \itshape
  {}% Indent amount
  {\itshape}% Theorem head font
  {.}% Punctuation after theorem head
  { }% Space after theorem head
  {}% Theorem head spec

\newtheorem{theorem}{Theorem}[subsection]

\newtheorem{corollary}[theorem]{Corollary}
\newtheorem{lemma}[theorem]{Lemma}
\newtheorem{proposition}[theorem]{Proposition}

\theoremstyle{definition}
\newtheorem{definition}[theorem]{Definition}

\theoremstyle{remarkitalic}
\newtheorem{remark}[theorem]{Remark}

\usepackage[babel,german=quotes]{csquotes}
\usepackage[style=numeric,sorting=none,backend=biber]{biblatex}
\bibliography{Literature}
\renewbibmacro{in:}{}
\renewbibmacro*{volume+number+eid}{%
  \printfield{volume}
  \addnbthinspace
  \printfield{number}%
  \printfield{eid}}
\DeclareFieldFormat[article]{number}{\mkbibparens{#1}}

 \numberwithin{equation}{section}
 
\makeatletter
\newcases{mycases}{~}{%
  \hfil$\m@th\displaystyle{##}$}{$\m@th\displaystyle{##}$\hfil}{\lbrace}{.}
\makeatother

\makeatletter
\AtBeginDocument{%
  \expandafter\renewcommand\expandafter\subsection\expandafter{%
    \expandafter\@fb@secFB\subsection
  }%
}
\makeatother

\makeatletter

\renewcommand\subsection{\@startsection{subsection}{2}{\z@}%
  {-3.25ex\@plus -1ex \@minus -.2ex}%
  {1.5ex \@plus .2ex}%
  {\normalfont\bfseries}}

\renewcommand\subsubsection{\@startsection{subsubsection}{3}{\z@}%
  {-3.0ex\@plus -1ex \@minus -.2ex}%
  {1.25ex \@plus .2ex}%
  {\normalfont\bfseries}}

\renewcommand\paragraph{\@startsection{paragraph}{4}{\z@}%
  {1.5ex \@plus .2ex \@minus .2ex}%
  {-1em}%
  {\normalfont\bfseries}}

\makeatother

\allowdisplaybreaks

\usepackage{fancyhdr}

\begin{document}

\makeatletter
\def\ps@pprintTitle{%
  \let\@oddfoot\@empty
  \let\@evenfoot\@empty
}
\makeatother

\fancypagestyle{plain}{
  \fancyhf{}
  \fancyfoot[C]{\thepage}
  %\fancyhead[C]{\textcolor{gray}{\footnotesize Preprint submitted to Simulation Modelling Practice and Theory}}
  \renewcommand{\headrulewidth}{0pt}
  \renewcommand{\footrulewidth}{0pt}
}

\fancypagestyle{firstpage}{
  \fancyhf{}
  %\fancyhead[C]{\textcolor{gray}{\footnotesize Preprint submitted to Insurance: Mathematics and Economics}}
  \renewcommand{\headrulewidth}{0pt}
}
	
\setcounter{page}{1}

\title{
Utility-Based Recovery Risk Measures:\\ Redundancy and Approximation for Portfolio Optimization
}

\author[1]{Justin Plückebaum%\corref{cor1}
}
\ead{justin.plueckebaum@insurance.uni-hannover.de}
\author[1]{Stefan Weber}
\ead{stefan.weber@insurance.uni-hannover.de}

\address[1]{Leibniz Universität Hannover, Institute of Actuarial and Financial Mathematics,\\ Welfengarten 1, 30167 Hannover, Germany\\ \text{}\\
\normalsize {October 09, 2026}}

\date{\today}

\begin{abstract}
 Recovery risk measures incorporate explicit liability-recovery requirements into asset--liability risk assessment. We study their specialization to optimized certainty equivalents (OCEs), allowing the underlying utility specification to vary across recovery levels and encompassing Recovery Average Value at Risk (RecAV@R) as a special case. For finite recovery specifications, we use the dual penalty representation to characterize redundant recovery regimes and derive position-independent screening conditions. For a subfamily of the Cressie--Read power-divergence family, including entropic and truncated-quadratic specifications, we obtain an exact characterization on atomless probability spaces: universal redundancy is equivalent to a two-dimensional convex-hull condition involving recovery levels and divergence-scale parameters. On such spaces, universal redundancy for RecAV@R reduces to pairwise dominance, whereas for power-divergence OCEs it can arise from the joint effect of several recovery regimes. For nonlinear OCE utilities, we develop a certified piecewise-linear approximation framework for portfolio optimization. In finite-scenario models, the approximating problems are finite-dimensional linear programs, while utility approximation errors transfer directly to the Recovery OCE, yielding explicit error bounds and convergence rates. Uniform approximation over the portfolio domain further implies convergence of optimal values and stability of optimal solutions. A numerical asset--liability study illustrates the effects of recovery-dependent risk preferences and the accuracy and convergence of the approximation.
\end{abstract}


\maketitle
\thispagestyle{firstpage}
\pagestyle{plain}
%\linenumbers

\section{Introduction}
Risk measures for asset–liability positions are commonly formulated in terms of the surplus remaining after liabilities have been met in full. While this perspective captures the risk of failing to achieve full solvency, it does not distinguish between different degrees of liability recovery in shortfall states. Recovery risk measures extend this perspective by imposing risk requirements at multiple recovery levels, assessing the position relative to different fractions of the liability rather than full liability coverage alone \cite{paperMunari}. They thereby account not only for the possibility of a shortfall, but also for its severity in terms of the fraction of liabilities that can still be covered. This distinction is particularly relevant in asset–liability management, where different degrees of underfunding may have substantially different economic consequences.

Within the recovery-risk framework, Recovery Average Value at Risk (RecAV@R) provides a particularly tractable specification and has been studied in both risk measurement and portfolio optimization \cite{paperMunari,paperPortfolioOptRecAV@R}. Its reliance on AV@R, however, ties the assessment at each recovery level to a specific tail-risk criterion. Optimized certainty equivalents (OCEs), introduced by Ben-Tal and Teboulle \cite{ben1986expected,paperBenTal}, provide a broader utility-based class of convex risk measures that includes AV@R as a special case and encodes risk preferences through the generating utility function. This suggests replacing the AV@R criterion at each recovery level by an OCE, with the generating utility allowed to depend on the recovery level. The resulting Recovery OCE framework combines recovery-sensitive risk assessment with the flexibility of OCEs, while retaining RecAV@R as a special case. This generalization, however, introduces both structural and computational challenges. When multiple recovery levels and risk specifications are imposed, it becomes important to determine which of these requirements effectively contribute to the aggregate risk measure and which are redundant. Computationally, while RecAV@R admits linear programming formulations in finite-scenario portfolio models, this tractability is generally lost for nonlinear OCE utilities, calling for suitable approximation schemes with explicit error control.

Our main theoretical contribution concerns the redundancy structure of finite Recovery OCE specifications. Building on general duality results for maxima of convex risk measures, we characterize the dual structure of finite Recovery OCEs in terms of the regime-wise penalty functions and the convexification arising from their combination. This yields general redundancy criteria and sufficient position-independent screening conditions. For the power-divergence subfamily, including entropic and truncated-quadratic specifications, we show that these screening conditions admit an exact geometric characterization on atomless probability spaces: universal redundancy is equivalent to a two-dimensional convex-hull condition involving recovery levels and divergence-scale parameters. On such spaces, we further show that universal redundancy for RecAV@R reduces to pairwise dominance. The power-divergence result therefore reveals a genuinely different phenomenon: a recovery regime may be redundant through the joint effect of several other regimes even when it is not dominated by any single one. We complement this structural analysis with a certified approximation framework for Recovery OCE portfolio optimization. Piecewise-linear approximations of the underlying OCE utilities yield finite-dimensional linear programming formulations in finite-scenario models. By localizing the scalar OCE optimizations to compact residual intervals, we obtain explicit approximation-error bounds and convergence rates under corresponding regularity assumptions. Uniform approximation over the portfolio domain further yields convergence of optimal values and stability of optimal portfolio solutions, with convergence of optimizers when the original problem has a unique optimizer.

The remainder of the paper is organized as follows. Section~2 develops the Recovery OCE framework and finite recovery specifications. Section~3 develops the redundancy analysis, from its dual characterization and general screening criteria to the exact geometric characterization for power-divergence OCEs. Section~4 develops the approximation framework, its linear programming formulation, and the corresponding approximation and stability results. Section~5 presents a numerical asset--liability study. Additional proofs are collected in the appendix.

\subsection{Related Literature}
Our work is most directly related to the literature on recovery risk measures. Munari et al. \cite{paperMunari} introduce recovery risk measures as a framework for incorporating explicit requirements on liability recovery into risk measurement and study Recovery Value at Risk and Recovery Average Value at Risk  as prominent specifications. Building on this framework, Munari, Plückebaum, and Weber \cite{paperPortfolioOptRecAV@R} study mean-risk portfolio selection under RecAV@R and derive tractable linear programming formulations for the nominal problem as well as under mixture and box uncertainty. Recovery OCEs, obtained by specifying OCE risk measures within the general class of recovery risk measures, were also considered by Duong \cite{Duong2025} in her master’s thesis, which provides an initial analysis of this particular class of recovery risk measures.  The present paper develops this line further by developing a redundancy theory for Recovery OCE regimes and establishing certified approximation and stability results for portfolio optimization.

A second strand of related work concerns convex risk measures and optimized certainty equivalents. Building on the axiomatic theory of coherent risk measures introduced by Artzner et al. \cite{artzner1999coherent}, the theory of convex risk measures and their dual representations in terms of penalty functions was developed by Frittelli and Rosazza Gianin \cite{frittelli2002putting} and Föllmer and Schied \cite{follmer2002convex}; see also Föllmer and Schied \cite{follmer2016stochastic} for a comprehensive treatment and Ruszczyński and Shapiro \cite{ruszczynski2006optimization} for an optimization perspective. A prominent utility-based class within this framework is provided by optimized certainty equivalents. Their underlying certainty-equivalent functional was introduced by Ben-Tal and Teboulle \cite{ben1986expected} and subsequently developed by the same authors as a class of convex risk measures \cite{paperBenTal}, including a dual representation in terms of convex divergence functionals. More recent work has developed extensions of optimized certainty equivalents in several directions, including robust formulations, preference uncertainty, stochastic control, and rank-dependent utility \cite{bartl2020computational,wu2022preference,veraguas2022stochastic,wu2024generalized}.

Particularly relevant to our dual analysis is the literature on combinations of convex risk measures. Föllmer and Schied \cite{follmer2002convex} show that the supremum of a family of convex risk measures admits a dual representation using the pointwise infimum of their penalty functions, while Righi \cite{righi2023theory} develops a broader theory of combinations of risk measures that includes worst-case combinations. These general results provide the duality basis for our analysis. We specialize them to finite Recovery OCEs and exploit the specific structure of the resulting OCE penalties to identify the minimal penalty, analyze when convexification is strict, characterize redundant recovery regimes, and derive position-independent screening criteria for redundancy.

A further strand of related literature concerns approximation methods for optimization problems involving nonlinear functions. Piecewise-linear approximation is a classical tool in approximation theory; see, for example, \cite{de1978practical} and \cite{powell1981approximation} for general treatments of piecewise-polynomial approximation. Approximation techniques are also well established in stochastic programming; see \cite{birge2011}.  Within utility-based optimization, piecewise-linear representations have been used to obtain tractable formulations of expected-utility and risk-based problems. In particular, Natarajan et al. \cite{natarajan2010tractable} develop robust expected-utility and risk models for portfolio optimization using piecewise-linear utility functions. More recently, piecewise-linear approximation has also been used in preference-robust optimization. Wu and Xu \cite{wu2022preference} study preference-robust modified OCEs and provide tractable approximations together with approximation-error bounds, while related approaches to optimization under utility ambiguity are developed by Guo et al. \cite{guo2024utility} and Wu et al. \cite{wu2025bi}. In contrast to the preference-robust literature, the OCE utilities in our setting are specified rather than ambiguous. We use piecewise-linear approximation to obtain tractable formulations for portfolio optimization with nonlinear Recovery OCEs and establish how utility-level approximation errors propagate to the Recovery OCE and the associated optimization problem, including convergence of optimal values and stability of optimal solutions.

\section{Utility-Based Recovery Risk Measures}
Optimized certainty equivalents provide a flexible utility-based framework for constructing convex monetary risk measures and encompass several standard risk measures, including Average Value at Risk, as special cases \cite{paperBenTal}. Combining the OCE framework with the general framework of recovery risk measures \cite{paperMunari}, we consider recovery-adjusted positions evaluated by OCE risk measures whose utility functions may vary across recovery levels. The resulting class of Recovery OCE (RecOCE) risk measures contains Recovery AV@R as a special case while allowing for more general recovery-dependent risk preferences. This section first recalls the OCE framework and then develops the properties of RecOCE risk measures needed for the subsequent analysis.

\subsection{Optimized Certainty Equivalent Risk Measures}
\label{subsec:oce}

Let $(\Omega,\mathcal F,\mathbb P)$ be a probability space and let \(L^\infty:=L^\infty(\Omega,\mathcal F,\mathbb P)\) denote the space of essentially bounded financial positions. Following Ben-Tal and Teboulle \cite{paperBenTal}, let \(\mathcal U_0\) denote  the class of closed, concave, non-decreasing functions $u:\mathbb R\to[-\infty,\infty)$ with nonempty effective domain satisfying
\[
    u(0)=0
    \qquad\text{and}\qquad
    1\in\partial u(0),
\]
where $\partial u$ denotes the superdifferential of $u$. By concavity and the normalization $u(0)=0$, the supergradient condition yields $u(x) \leq x$ for all $x\in\mathbb R$.

\begin{definition}
\label{def:oce}
Let \(u\in\mathcal U_0\). The \emph{optimized certainty equivalent}
associated with \(u\) is defined by
\[
    S_u(X)
    :=
    \sup_{m\in\mathbb R}
    \left\{
        m+\mathbb E[u(X-m)]
    \right\},
    \qquad X\in L^\infty.
\]
\end{definition}
The OCE evaluates a risky position \(X\) by optimally decomposing it into a deterministic cash amount \(m\) and a residual position \(X-m\). The cash amount is valued linearly, whereas the residual position is evaluated through the utility function \(u\).

The following lemma provides a sufficient condition under which the scalar optimization problem defining the OCE can be localized to the essential range of the position.
\begin{lemma}
\label{lemma:oce-sup-attained}
Let \(u\in\mathcal U_0\) and \(X\in L^\infty\), and set
\[
    x_{\min}:=\operatorname{ess\,inf}X,
    \qquad
    x_{\max}:=\operatorname{ess\,sup}X.
\]
If \(u\) is finite on $[x_{\min}-x_{\max},\,x_{\max}-x_{\min}]$, then the supremum defining \(S_u(X)\) is attained, and
\[
    S_u(X)
    =
    \max_{m\in[x_{\min},x_{\max}]}
    \left\{
        m+\mathbb E[u(X-m)]
    \right\}.
\]
\end{lemma}

\begin{proof}
The proof is given in \Cref{app:oce-sup-attained}.
\end{proof}

In particular, the optimization may be restricted to the essential range of \(X\), in which case
\[
X-m \in [x_{\min}-x_{\max},\,x_{\max}-x_{\min}]
\qquad \mathbb P\text{-a.s.}
\] 
Thus, the subsequent approximation analysis requires approximating $u$ only on this compact residual interval.

By Ben-Tal and Teboulle \cite[Theorem~2.1]{paperBenTal}, the OCE \(S_u\) is translation invariant, monotone, and concave. Hence, whenever \(S_u\) is finite-valued on \(L^\infty\),
\[
    \rho_u(X):=-S_u(X), \qquad X\in L^\infty,
\]
defines a convex monetary risk measure, which we call the utility-based risk measure generated by \(u\). The OCE framework thus provides a flexible class of convex monetary risk measures through the choice of the underlying utility function. Two important examples are the entropic risk measure and Average Value at Risk.

\begin{remark}
\label{ex:standard-oce-risk-measures} \text{ }
\begin{enumerate}

    \item \textbf{\emph{Entropic risk measure.}}
    For \(\theta>0\), the exponential utility
    \(u_\theta(x)=\theta(1-e^{-x/\theta})\) induces
    \[
        \rho_{u_\theta}(X)
        =
        \theta\log\mathbb E\!\left[e^{-X/\theta}\right],
    \]
    the entropic risk measure with absolute risk-aversion parameter
    \(1/\theta\).

    \item \textbf{\emph{Average Value at Risk.}}
    Let \(\alpha\in(0,1)\) and consider the piecewise-linear utility function
    \[
        u_\alpha(x)
        =
        \begin{cases}
            x/\alpha, & x\leq 0,\\
            0,        & x>0.
        \end{cases}
    \]
    The associated utility-based risk measure is
    \[
        \rho_{u_\alpha}(X)
        =
        \inf_{m\in\mathbb R}
        \left\{
            \frac{1}{\alpha}\mathbb E[(m-X)^+]-m
        \right\}
        =
        \operatorname{AV@R}_\alpha(X),
    \]
    where the last equality follows from the Rockafellar--Uryasev representation \cite{rockafeller2000,rockafeller2002}. Here, $\alpha$ denotes the tail probability, so that $\operatorname{AV@R}_\alpha$ corresponds to the confidence level $1-\alpha$, with X interpreted as a financial position rather than as a loss variable. 
\end{enumerate}
\end{remark}

\subsection{Recovery OCE Risk Measures}
We specialize the recovery risk measure framework of Munari et al. \cite{paperMunari} to OCE risk measures by allowing the generating utility function to vary across recovery levels.

Let $\mathcal A\subseteq\mathbb R$ be a parameter space and consider a family $\mathfrak u=(u_\alpha)_{\alpha\in\mathcal A}$ of utility functions in $\mathcal U_0$ such that the associated OCEs are finite-valued on $L^\infty$. We assume that $\mathfrak u$ is pointwise nondecreasing in $\alpha$, that is,
\[
    \alpha_1\leq\alpha_2
    \quad\Longrightarrow\quad
    u_{\alpha_1}(x)\leq u_{\alpha_2}(x)
    \qquad\text{for all }x\in\mathbb R.
\]
Since the OCE is monotone with respect to the generating utility function, this ordering is reversed for the associated risk measures:
\[
    \alpha_1\leq\alpha_2
    \quad\Longrightarrow\quad
    \rho_{u_{\alpha_1}}(X)\geq\rho_{u_{\alpha_2}}(X)
    \qquad\text{for all }X\in L^\infty.
\]

\begin{definition}
\label{def:recoce} Let $\gamma:[0,1]\to\mathcal A$ be a nondecreasing level function. \\
The \emph{Recovery OCE (RecOCE) risk measure} associated with $\mathfrak u$ and $\gamma$ is defined by
\[
    \operatorname{Rec}\rho_{\mathfrak u,\gamma}(X,Y)
    :=
    \sup_{\lambda\in[0,1]}
    \rho_{u_{\gamma(\lambda)}}
    \bigl(X+(1-\lambda)Y\bigr),
    \qquad (X,Y)\in L^\infty\times L^\infty.
\]
\end{definition}

\begin{remark}[Recovery AV@R]
\label{ex:recavar-as-recoce}
Consider the AV@R utilities \(u_\alpha\), \(\alpha\in(0,1)\), from \Cref{ex:standard-oce-risk-measures}. The family \(\mathfrak u=(u_\alpha)_{\alpha\in(0,1)}\) is pointwise nondecreasing in \(\alpha\), and \(\rho_{u_\alpha}=\operatorname{AV@R}_\alpha\). Hence, for any nondecreasing level function \(\gamma:[0,1]\to(0,1)\),
\[
\begin{aligned}
    \operatorname{Rec}\rho_{\mathfrak u,\gamma}(X,Y)
    &=
    \sup_{\lambda\in[0,1]}
    \operatorname{AV@R}_{\gamma(\lambda)}
    \bigl(X+(1-\lambda)Y\bigr) \\
    &=
    \operatorname{RecAV@R}_{\gamma}(X,Y).
\end{aligned}
\]
Thus, the established Recovery AV@R risk measure \cite{paperMunari,paperPortfolioOptRecAV@R} is recovered as a special case of the RecOCE framework.\\
\end{remark}

\begin{remark}
Three aspects of the definition are worth emphasizing.
\begin{itemize}
    \item[(i)] The level function \(\gamma\) determines how the parameter indexing the family of OCE risk measures varies across recovery levels. The term ``level'' is used in a generalized sense: depending on the underlying utility family, the parameter may represent a tail level, a risk-tolerance parameter, or another preference parameter. Since both \(\gamma\) and the utility family are nondecreasing, lower recovery levels are associated with weakly more conservative risk specifications.

    \item[(ii)] The two arguments admit a natural asset--liability interpretation.  For assets \(A\in L^\infty\) and liabilities \(L\in L^\infty_+\), setting \(X=A-L\) and \(Y=L\) gives
    \[
        X+(1-\lambda)Y=A-\lambda L.
    \]
    Hence, \(\lambda\) represents the fraction of liabilities required to be covered, and the RecOCE evaluates the corresponding recovery-adjusted position using the OCE risk specification selected by
    \(\gamma(\lambda)\).

    \item[(iii)] Under this asset--liability interpretation, the acceptance condition can be written as
    \[
    \operatorname{Rec}\rho_{\mathfrak u,\gamma}(A-L,L)\leq0
    \quad\Longleftrightarrow\quad
    \rho_{u_{\gamma(\lambda)}}(A-\lambda L)\leq0
    \quad\text{for all }\lambda\in[0,1].
    \]
    Thus, acceptability is required simultaneously across all recovery-adjusted positions. In the RecAV@R case, these conditions imply the recovery-probability bounds
    \[
        \mathbb P(A\geq\lambda L)\geq1-\gamma(\lambda),
        \qquad \lambda\in[0,1],
    \]
    as shown by Munari et al. \cite{paperMunari}. General OCE risk measures need not admit an analogous probability-based interpretation; instead, their acceptance conditions are determined by the utility functions associated with the respective recovery levels. Recovery OCEs therefore trade the direct probability-based interpretation of RecAV@R for greater flexibility in specifying recovery-dependent risk preferences.
\end{itemize}
\end{remark}

The basic structural properties of Recovery OCE risk measures are collected in the following proposition.

\begin{proposition}
\label{prop:properties-recrho}
The Recovery OCE risk measure
\(\operatorname{Rec}\rho_{\mathfrak u,\gamma}\) is cash invariant in its
first component, monotone decreasing in both arguments and jointly convex in \((X,Y)\). Consequently, for each fixed \(Y\in L^\infty\), the mapping \(X\mapsto\operatorname{Rec}\rho_{\mathfrak u,\gamma}(X,Y)\) is a convex monetary risk measure.
If \(Y\in L^\infty_+\), then
\[
    \operatorname{Rec}\rho_{\mathfrak u,\gamma}(0,Y)=0,
    \qquad
    \operatorname{Rec}\rho_{\mathfrak u,\gamma}(X,Y)
    \leq \rho_{u_{\gamma(0)}}(X),
    \qquad X\in L^\infty.
\]
\end{proposition}

\begin{proof}
See \Cref{app:proof-properties-recrho}.
\end{proof}

\subsection{Finite Recovery Specifications}

Rather than specifying a risk preference for every recovery level \(\lambda\in[0,1]\), a decision maker may distinguish only finitely many economically relevant recovery regimes. We formalize such requirements through a finite recovery specification.

Let \(\mathfrak u=(u_\alpha)_{\alpha\in\mathcal A}\) be a pointwise nondecreasing utility family as in \Cref{def:recoce}. Consider recovery levels
\[
    0=r_0<r_1<\cdots<r_l<r_{l+1}=1
\]
and parameters
\[
    \alpha_1\leq\alpha_2\leq\cdots\leq\alpha_{l+1},
    \qquad \alpha_j\in\mathcal A.
\]
The associated finite recovery specification is represented by the stepwise level function
\begin{align}\label{eq:finiteRecoverySpecification}
    \gamma(\lambda)
    =
    \begin{cases}
        \alpha_1, & r_0\leq\lambda\leq r_1,\\
        \alpha_2, & r_1<\lambda\leq r_2,\\
        \vdots\\
        \alpha_{l+1}, & r_l<\lambda\leq r_{l+1}.
    \end{cases}
\end{align}
The thresholds \(r_j\) delimit recovery regimes with distinct risk specifications, while \(\alpha_j\) determines the OCE risk measure associated with the \(j\)-th regime. As in the finite-recovery reduction of \cite{paperMunari}, monotonicity yields an exact finite representation for RecOCE.

\begin{proposition}[Finite recovery representation]
\label{prop:finite-recovery-representation}
Let \(\gamma\) be given by
\eqref{eq:finiteRecoverySpecification} and let
\(Y\in L^\infty_+\). Then, for every \(X\in L^\infty\),
\[
    \operatorname{Rec}\rho_{\mathfrak u,\gamma}(X,Y)
    =
    \max_{j=1,\ldots,l+1}
    \rho_{u_{\alpha_j}}
    \bigl(X+(1-r_j)Y\bigr).
\]
\end{proposition}
\begin{proof}
Fix \(j\in\{1,\ldots,l+1\}\). On the \(j\)-th recovery regime,
\(\gamma(\lambda)=\alpha_j\). Since \(Y\geq0\), the mapping $\lambda\longmapsto X+(1-\lambda)Y$ is nonincreasing. By monotonicity of \(\rho_{u_{\alpha_j}}\), the mapping $\lambda\longmapsto \rho_{u_{\alpha_j}}\bigl(X+(1-\lambda)Y\bigr)$ is therefore nondecreasing. Hence its supremum over the \(j\)-th regime is attained at the right endpoint \(r_j\). Taking the maximum over all regimes yields the asserted representation.
\end{proof}

\subsection{Portfolio Optimization under Recovery OCE Risk Measures}

We formulate the Recovery OCE portfolio problem in a finite-scenario setting. Let \(b>0\) denote the investor's initial budget and let \(\mathbf R_s\in\mathbb R^K\), \(s=1,\ldots,S\), denote simple-return scenarios with probabilities \(p_s>0\), where \(\sum_{s=1}^S p_s=1\). For portfolio weights \(\mathbf w\in\mathbb R^K\), the terminal asset value in scenario \(s\) is \(A_s(\mathbf w):=b(1+\mathbf w^\top\mathbf R_s)\), and we write \(\mathbf A(\mathbf w):=(A_1(\mathbf w),\ldots,A_S(\mathbf w))\). Further, let \(\mathbf L:=(L_1,\ldots,L_S)\in\mathbb R_+^S\) denote the scenario-dependent liability vector, which is independent of the portfolio choice. Short selling is excluded, so admissible portfolio weights belong to the unit simplex \(\Delta_K:=\{\mathbf w\in\mathbb R_+^K: \sum_{k=1}^K w_k=1\}\).

Following the mean--risk portfolio formulation studied for Recovery AV@R by Munari et al. \cite{paperPortfolioOptRecAV@R}, we minimize the RecOCE risk subject to a lower bound on expected portfolio return. For a target expected return \(\mu\in\mathbb R\), define
\[
    \Delta_K(\mu)
    :=
    \left\{
        \mathbf w\in\Delta_K:
        \sum_{s=1}^S p_s\,\mathbf w^\top\mathbf R_s
        \geq \mu
    \right\}.
\]
As a closed subset of the compact simplex \(\Delta_K\), \(\Delta_K(\mu)\) is compact; throughout, we assume that it is nonempty.

The resulting Recovery OCE portfolio problem is $\inf_{\mathbf w\in\Delta_K(\mu)} \operatorname{Rec}\rho_{\mathfrak u,\gamma} \bigl(\mathbf A(\mathbf w)-\mathbf L,\mathbf L\bigr)$. For the finite recovery specification introduced in the preceding subsection, the exact finite representation yields
\begin{equation}
\label{eq:finiteRecOCEPortfolioProblem}
    \inf_{\mathbf w\in\Delta_K(\mu)}
    \max_{j=1,\ldots,l+1}
    \rho_{u_{\alpha_j}}
    \bigl(\mathbf A(\mathbf w)-r_j\mathbf L\bigr).
\end{equation}

Problem~\eqref{eq:finiteRecOCEPortfolioProblem} provides the reference portfolio formulation for the subsequent analysis. Section~3 studies when individual recovery regimes in the maximum are redundant, while Section~4 develops the approximation framework for the resulting portfolio problem.

\section{Dual Structure and Redundancy of Recovery Regimes}
\label{sec:dual}

We first briefly specialize standard convex duality to finite Recovery OCEs and derive the penalty representation needed for the subsequent analysis. Building on this structure, we characterize redundant recovery regimes and develop position-independent screening conditions, culminating in an exact geometric characterization for power-divergence OCEs.

\subsection{Dual Structure and Convexification Across Recovery Regimes}

Throughout this subsection, we fix \(Y\in L^\infty_+\) and consider the finite recovery specification \eqref{eq:finiteRecoverySpecification}. By \Cref{prop:finite-recovery-representation}, the corresponding Recovery OCE is
\[
\operatorname{Rec}\rho_{\mathfrak u,\gamma}(X,Y)
=
\max_{j=1,\ldots,l+1}
\rho_{u_{\alpha_j}}\bigl(X+(1-r_j)Y\bigr).
\]
We first identify the regime-wise penalty functions and then determine the minimal penalty associated with a finite recovery specification.

Here,  we use the standard duality for convex monetary risk measures on $L^\infty$; see, e.g., \cite{follmer2002convex, follmer2016stochastic}. Let $\mathcal D := \left\{ Z\in L^1_+:\mathbb E[Z]=1 \right\},$ which we identify with the set of probability measures absolutely continuous with respect to $\mathbb P$ through their Radon--Nikodym densities. For a convex monetary risk measure $\rho:L^\infty\to\mathbb R$ with the Fatou property, the minimal penalty function is
\[
\alpha_\rho^{\min}(Z)
:=
\sup_{X\in L^\infty}
\left\{
\mathbb E[-ZX]-\rho(X)
\right\},
\qquad Z\in\mathcal D.
\]
Moreover,
\[
\rho(X)
=
\sup_{Z\in\mathcal D}
\left\{
\mathbb E[-ZX]-\alpha_\rho^{\min}(Z)
\right\},
\qquad X\in L^\infty.
\]

We restrict attention to OCE utilities generated by divergence functions \(\phi\) belonging to the class \(\Phi\) of \cite{paperBenTal} whose effective domains are intervals contained in \([0,\infty)\). Thus, we consider proper, closed, convex functions \(\phi:\mathbb R\to(-\infty,\infty]\) satisfying \(1\in\operatorname{int}\operatorname{dom}\phi\) and \(\phi(1)=\min_{z\in\mathbb R}\phi(z)=0\). For each \(j=1,\ldots,l+1\), let \(\phi_{\alpha_j}\in\Phi\) and define \(u_{\alpha_j}(x):=-\phi_{\alpha_j}^*(-x)\). Since \(\operatorname{dom}\phi_{\alpha_j}\subseteq[0,\infty)\) and \(\phi_{\alpha_j}(1)=\min\phi_{\alpha_j}=0\), standard convex-conjugacy arguments show that \(u_{\alpha_j}\) is nondecreasing and concave, satisfies \(u_{\alpha_j}(0)=0\), and \(1\in\partial u_{\alpha_j}(0)\). Hence \(u_{\alpha_j}\in\mathcal U_0\).

By \cite[Theorems~4.2 and~4.4]{paperBenTal}, the OCE risk measure
associated with \(u_{\alpha_j}\) has minimal penalty
\(\alpha_{\rho_{u_{\alpha_j}}}^{\min}(Z)
=\mathbb E[\phi_{\alpha_j}(Z)]\), \(Z\in\mathcal D\).

For \(j=1,\ldots,l+1\), define
\(\rho_{j,Y}(X):=\rho_{u_{\alpha_j}}(X+(1-r_j)Y)\).
The liability adjustment in each recovery regime modifies the corresponding
OCE penalty by a linear term. Indeed, a translation in the definition of
the minimal penalty yields
\[
g_{j,Y}(Z)
:=
\alpha_{\rho_{j,Y}}^{\min}(Z)
=
\mathbb E[\phi_{\alpha_j}(Z)]
+
(1-r_j)\mathbb E[ZY],
\qquad Z\in\mathcal D.
\]
Consequently, setting
\[
\beta_Y(Z)
:=
\min_{j=1,\ldots,l+1}g_{j,Y}(Z),
\qquad Z\in\mathcal D,
\]
we obtain
\[
\begin{aligned}
\operatorname{Rec}\rho_{\mathfrak u,\gamma}(X,Y)
&=
\max_{j=1,\ldots,l+1}
\sup_{Z\in\mathcal D}
\left\{
\mathbb E[-ZX]-g_{j,Y}(Z)
\right\}\\
&=
\sup_{Z\in\mathcal D}
\left\{
\mathbb E[-ZX]-\beta_Y(Z)
\right\}.
\end{aligned}
\]
Since the right-hand side is a pointwise supremum of \(\sigma(L^\infty,L^1)\)-continuous affine functions of \(X\), the Recovery OCE is \(\sigma(L^\infty,L^1)\)-lower semicontinuous and therefore has the Fatou property. The representing penalty \(\beta_Y\) need not be convex and therefore need not coincide with the minimal penalty function. We extend \(\beta_Y\) to \(L^1\) by setting \(\beta_Y(Z)=+\infty\) for \(Z\notin\mathcal D\).

\begin{proposition}
Let $\gamma$ be the stepwise level function associated with the finite
recovery specification \eqref{eq:finiteRecoverySpecification}, and fix
$Y\in L^\infty_+$. Then the minimal penalty function of
\[
X\longmapsto
\operatorname{Rec}\rho_{\mathfrak u,\gamma}(X,Y)
\]
is the $\sigma(L^1,L^\infty)$-lower semicontinuous convex envelope of
$\beta_Y$. More precisely,
\[
\alpha_{\operatorname{Rec}\rho_{\mathfrak u,\gamma}(\cdot,Y)}^{\min}
=
\operatorname{cl}_{\sigma(L^1,L^\infty)}
\operatorname{conv}\beta_Y.
\]

\end{proposition}

\begin{proof}
By the preceding representation and the extension of \(\beta_Y\) to
\(L^1\),
\[
\operatorname{Rec}\rho_{\mathfrak u,\gamma}(X,Y)
=
\sup_{Z\in L^1}
\left\{
\mathbb E[-ZX]-\beta_Y(Z)
\right\}
=
\beta_Y^*(-X),
\]
where the conjugate is taken with respect to the dual pair
\((L^1,L^\infty)\) with pairing
\(\langle Z,X\rangle=\mathbb E[ZX]\).
Hence, by the definition of the minimal penalty and substitution $U=-X$,
\[
\begin{aligned}
\alpha_{\operatorname{Rec}\rho_{\mathfrak u,\gamma}(\cdot,Y)}^{\min}(Z)
&=
\sup_{X\in L^\infty}
\left\{
\mathbb E[-ZX]
-
\operatorname{Rec}\rho_{\mathfrak u,\gamma}(X,Y)
\right\} \\
&=
\sup_{U\in L^\infty}
\left\{
\mathbb E[ZU]-\beta_Y^*(U)
\right\} \\
&=
\beta_Y^{**}(Z).
\end{aligned}
\]
By the Fenchel--Moreau biconjugation theorem, $\beta_Y^{**}$ is the $\sigma(L^1,L^\infty)$-lower semicontinuous convex envelope of $\beta_Y$. Therefore,
\[
\alpha_{\operatorname{Rec}\rho_{\mathfrak u,\gamma}(\cdot,Y)}^{\min}
=
\beta_Y^{**}
=
\operatorname{cl}_{\sigma(L^1,L^\infty)}
\operatorname{conv}\beta_Y.
\]

\end{proof}

The convexification in the preceding representation can be strict. Indeed, even for two recovery regimes generated by standard entropic OCEs, the pointwise minimum of the regime-wise minimal penalty functions may fail to be convex. \Cref{app:strict-convexification} provides an explicit two-state example. Thus, the convexification in the minimal penalty representation cannot in general be omitted.

%%%%%%%%%%%%%%%%%%%%%%%%%%%%%%%%%%%%%%%%%%%%%%%%
%%%%%%%%%%%%%%%%%%%%%%%%%%%%%%%%%%%%%%%%%%%%%%%%
%%%%%%%%%%%%%%%%%%%%%%%%%%%%%%%%%%%%%%%%%%%%%%%%

\subsection{Effective and Redundant Recovery Regimes}
\label{sec:redundancy}

For the finite recovery specification, write
\[
R_{j,Y}(X)
:=
\rho_{u_{\alpha_j}}\bigl(X+(1-r_j)Y\bigr),
\qquad j=1,\ldots,l+1.
\]
Then
\[
\operatorname{Rec}\rho_{\mathfrak u,\gamma}(X,Y)
=
\max_{j=1,\ldots,l+1}R_{j,Y}(X).
\]
A natural question is whether each recovery requirement affects the resulting risk measure, or whether some requirements are already implied by the remaining ones. This motivates the following notion.

\begin{definition}
Fix \(Y\in L^\infty_+\), \(j\in\{1,\ldots,l+1\}\), and define $R_{- j,Y}(X) := \max_{k\neq j}R_{k,Y}(X)$. Regime \(j\) is called \emph{redundant for \(Y\)} if
\[
R_{j,Y}(X)\leq R_{-j,Y}(X)
\qquad
\text{for all }X\in L^\infty.
\]
Otherwise, regime $j$ is called \emph{effective for Y} if  \(R_{j,Y}(X)>R_{-j,Y}(X)\) for some \(X\in L^\infty\), that is, if it is uniquely active at some position.
\end{definition}

We first characterize redundancy for a fixed liability position in terms of the regime-wise minimal penalties. Recall that \(g_{j,Y}\) is the minimal penalty of \(R_{j,Y}\). Let \(\alpha_{-j,Y}^{\min}\) denote the minimal penalty of \(R_{-j,Y}\). By the convexification result above,
\[
\alpha_{-j,Y}^{\min}
=
\operatorname{cl}_{\sigma(L^1,L^\infty)}
\operatorname{conv}
\left(
\min_{k\neq j}g_{k,Y}
\right),
\]
where the pointwise minimum is extended by \(+\infty\) outside
\(\mathcal D\).

\begin{proposition}\label{proposition:redundancyDual}
Fix \(Y\in L^\infty_+\) and \(j\in\{1,\ldots,l+1\}\). Then regime \(j\)
is redundant for \(Y\) if and only if
\[
g_{j,Y}(Z)
\geq
\alpha_{-j,Y}^{\min}(Z)
\qquad
\text{for all }Z\in\mathcal D.
\]
\end{proposition}

\begin{proof}
Redundancy is equivalent to
\(R_{j,Y}\leq R_{-j,Y}\) on \(L^\infty\). By order reversal between
convex risk measures and their minimal penalty functions, this is equivalent
to
\[
\alpha_{R_{j,Y}}^{\min}
\geq
\alpha_{-j,Y}^{\min}.
\]
Since \(\alpha_{R_{j,Y}}^{\min}=g_{j,Y}\), the assertion follows.
\end{proof}

The exact criterion in \Cref{proposition:redundancyDual} involves the convexified penalty of all remaining regimes. A simpler sufficient condition is pairwise dominance: for \(j\neq k\),
\[
R_{j,Y}\leq R_{k,Y}\ \text{on }L^\infty
\quad\Longleftrightarrow\quad
g_{j,Y}\geq g_{k,Y}\ \text{on }\mathcal D.
\]
Hence, dominance by a single regime implies redundancy, but is not necessary: redundancy may instead result from the combined effect of several regimes.\\

The preceding criteria concern redundancy for a fixed liability position \(Y\). For the design of a recovery specification, however, it is useful to identify regimes that can be removed independently of the particular asset--liability position. This leads to the following stronger notion.

\begin{definition}
A recovery regime \(j\) is called \emph{universally redundant} if it is redundant for every \(Y\in L^\infty_+\).
\end{definition}

\begin{proposition}[Structural screening of recovery regimes]
\label{proposition:structuralScreening}
Let $j\in\{1,\ldots,l+1\}$, and let $r_1,\ldots,r_{l+1}$ denote the recovery levels of the finite recovery specification. Suppose that there exist weights $\xi_k\geq0$, $k\neq j$, with $\sum_{k\neq j}\xi_k=1$ such that
\[
    r_j\leq\sum_{k\neq j}\xi_k r_k
\quad\text{ and } \quad
    \phi_{\alpha_j}(z)
    \geq
    \sum_{k\neq j}\xi_k\phi_{\alpha_k}(z)
    \qquad
    \text{for every }z\geq0.
\]
Then regime $j$ is universally redundant.
\end{proposition}

\begin{proof}
Fix \(Y\in L^\infty_+\) and \(Z\in\mathcal D\). Then
\[
\begin{aligned}
g_{j,Y}(Z)
&=
\mathbb E[\varphi_{\alpha_j}(Z)]
+(1-r_j)\mathbb E[ZY]\\
&\geq
\sum_{\substack{k\ne j\\ \xi_k>0}}
\xi_k\mathbb E[\varphi_{\alpha_k}(Z)]
+
\sum_{\substack{k\ne j\\ \xi_k>0}}
\xi_k(1-r_k)\mathbb E[ZY]\\
&=
\sum_{\substack{k\ne j\\ \xi_k>0}}
\xi_k g_{k,Y}(Z)\\
&\geq
\min_{k\ne j}g_{k,Y}(Z)
\geq
\alpha^{\min}_{-j,Y}(Z).
\end{aligned}
\]
\Cref{proposition:redundancyDual} therefore implies that regime \(j\)
is redundant for \(Y\). Since \(Y\in L^\infty_+\) was arbitrary, regime
\(j\) is universally redundant.
\end{proof}

\Cref{proposition:structuralScreening} yields a position-independent screening criterion based solely on the recovery levels \((r_k)\) and the regime-wise divergence functions \((\phi_{\alpha_k})\). Its sharpness, however, depends on the structure of the underlying divergence family. We first consider RecAV@R, for which universal redundancy admits a particularly simple exact characterization.

For AV@R, the regime-wise divergence function is
\[
    \phi_{\alpha}(z)
    =
    \begin{cases}
        0, & 0\leq z\leq 1/\alpha,\\
        +\infty, & \text{otherwise},
    \end{cases}
\]
and the corresponding dual set is $\mathcal D_\alpha = \left\{ Z\in L^1_+: \mathbb E[Z]=1,\; Z\leq 1/\alpha\ \text{a.s.} \right\}$. Universal redundancy then reduces exactly to pairwise dominance.

\begin{corollary}
\label{cor:recavar-redundancy}
Consider a finite RecAV@R specification with regimes \((r_k,\alpha_k)\), \(k=1,\ldots,l+1\), and fix \(j\). Suppose that there exists \(A\in\mathcal F\) such that
\(\mathbb P(A)=\alpha_j\). Then regime \(j\) is universally redundant if
and only if there exists \(k\neq j\) such that
\[
    r_k\geq r_j,
    \qquad
    \alpha_k\leq\alpha_j.
\]
In particular, this characterization holds for every regime on an atomless
probability space.
\end{corollary}

\begin{proof}
See \Cref{app:recavar-redundancy}.
\end{proof}

On an atomless probability space, universal redundancy in RecAV@R reduces to pairwise dominance: a regime cannot become universally redundant solely through the combined effect of several individually non-dominating regimes. Under the ordered finite recovery specification \eqref{eq:finiteRecoverySpecification}, regime \(j<l+1\) is universally redundant if and only if $\alpha_{j+1}=\alpha_j$. In particular, if the AV@R levels are strictly increasing, no recovery regime is universally redundant.

A different structure arises for common-scale divergence families, including the entropic family. Suppose that
\[
\phi_{\theta}(z)=\theta\phi(z),
\qquad \theta>0,
\]
for some \(\phi\in\Phi\). Since \(\phi\geq0\), the conditions in \Cref{proposition:structuralScreening} imply that regime \(j\) is universally redundant whenever the convex hull of the remaining regime parameter pairs \(\{(r_k,\theta_k):k\neq j\}\) contains a point \((\bar r,\bar\theta)\) satisfying \(\bar r\geq r_j\) and \(\bar\theta\leq\theta_j\). For general common-scale families, this condition is only sufficient. For the power-divergence family introduced below, we show that it is also necessary on atomless probability spaces.

\begin{definition}
For \(p\geq1\), define the power-divergence generator
\[
\phi_p(z)
=
\begin{cases}
z\log z-z+1, & p=1,\\[1mm]
\dfrac{z^p-pz+p-1}{p(p-1)}, & p>1,
\end{cases}
\qquad z\geq0,
\]
with \(0\log0:=0\), and set \(\phi_p(z)=+\infty\) for \(z<0\).
These functions form a subfamily of the Cressie--Read power-divergence
family \cite{cressie1984multinomial}. For fixed \(p\geq1\), define the associated common-scale family by
\[
\phi_{\theta,p}(z):=\theta\phi_p(z),
\qquad \theta>0.
\]
We refer to the corresponding OCEs as \emph{power-divergence OCEs}.
\end{definition}
The cases \(p=1\) and \(p=2\) correspond to the entropic and truncated-quadratic OCEs, respectively, with \(\phi_{\theta,1}(z)=\theta(z\log z-z+1)\) and \(\phi_{\theta,2}(z)=\frac{\theta}{2}(z-1)^2\).

\begin{theorem}[Exact characterization of universal redundancy]
\label{theorem:powerUniversalRedundancy}
Assume that \((\Omega,\mathcal F,\mathbb P)\) is atomless and fix \(p\geq1\). Consider power-divergence OCE regimes with parameters \((r_k,\theta_k)\), \(k=1,\ldots,l+1\), where \(r_k\in[0,1]\) and \(\theta_k>0\), each generated by $\phi_{\theta_k,p}(z)=\theta_k\phi_p(z)$. Set \mbox{\(q_k:=(\theta_k,1-r_k)\)}. For \(j\in\{1,\ldots,l+1\}\), the following statements are equivalent:
\begin{enumerate}
    \item[(i)] Regime \(j\) is universally redundant.

    \item[(ii)]
    $q_j\in\operatorname{conv}\{q_k:k\neq j\}+\mathbb R_+^2$.

    \item[(iii)] There exist weights \(\xi_k\geq0\), \(k\neq j\), such that
    \[
    \sum_{k\neq j}\xi_k=1,\qquad
    \theta_j\geq\sum_{k\neq j}\xi_k\theta_k,\qquad
    r_j\leq\sum_{k\neq j}\xi_k r_k.
    \]
\end{enumerate}
\end{theorem}

\begin{proof}[Proof idea]
The equivalence of \textup{(ii)} and \textup{(iii)} is immediate from the definition of $q_k$, while \textup{(iii)}$\Rightarrow$\textup{(i)} follows from \Cref{proposition:structuralScreening}, since $\phi_{\theta,p}=\theta\phi_p$ and $\phi_p\geq0$.

For the converse, suppose that \textup{(ii)} fails. Strict separation of \(q_j\) from \(\operatorname{conv}\{q_k:k\neq j\}+\mathbb R_+^2\) yields $h,a\geq0$, not both zero, such that
\begin{equation}
    h\theta_j+a(1-r_j)
    <
    h\theta_k+a(1-r_k),
    \qquad k\neq j.
    \tag{*}
\end{equation}
If $h=0$, a positive deterministic liability already makes regime $j$ strictly active.

Let $h>0$. We construct increasingly severe losses on increasingly rare events. For $p=1$, choose
\[
    \mathbb P(A_t)=e^{-th},
    \qquad
    X_t=-tx\,\mathbf 1_{A_t},
    \qquad
    Y_t=ta\,\mathbf 1_{A_t}.
\]
For sufficiently large $x$, the entropic representation gives
\[
    \frac{1}{t}R_{k,Y_t}(X_t)
    \longrightarrow
    x-a(1-r_k)-h\theta_k.
\]
For $p>1$, the same construction with
\[
    \mathbb P(A_t)=[p(p-1)ht]^{-1/(p-1)}
\]
yields
\[
    \frac{1}{t}R_{k,Y_t}(X_t)
    \longrightarrow
    \sup_{q\in[0,1]}
    \{q[x-a(1-r_k)]-h\theta_kq^p\}.
\]
Choosing
\[
    x>\max_k\{a(1-r_k)+ph\theta_k\}
\]
ensures that the supremum is attained at $q=1$, so that exactly the same limiting expression $x-a(1-r_k)-h\theta_k$ is obtained.

Thus, in both cases, the limiting risk is $x$ minus the separating functional in \textup{(*)}. Hence its value is strictly largest for regime $j$. It follows that regime $j$ is strictly active for all sufficiently large $t$ and therefore cannot be universally redundant. This proves the converse by contraposition. Full details are given in \Cref{app:powerUniversalRedundancy}.
\end{proof}

\begin{remark}
Failure of the geometric condition implies more than non-redundancy: the construction in the proof yields \(Y\in L^\infty_+\) and \(X\in L^\infty\) such that \(R_{j,Y}(X)>\max_{k\neq j}R_{k,Y}(X)\). Thus, regime \(j\) is uniquely active at some asset--liability position. Moreover, holding the remaining regimes fixed, the geometric characterization directly identifies the parameter region in which a given regime is universally redundant.

Atomlessness is needed only for the necessity and strict-activity parts of \Cref{theorem:powerUniversalRedundancy}. The sufficient screening criterion in \Cref{proposition:structuralScreening} remains valid on arbitrary probability spaces, including finite discrete scenario spaces.
\end{remark}

On atomless probability spaces, the dual analysis reveals a qualitative difference between RecAV@R and power-divergence OCEs: universal redundancy reduces to pairwise dominance in the former, whereas it is governed by the joint convex geometry of the recovery regimes in the latter. These results provide structural criteria for screening redundant regimes before optimization. After such structural screening, the remaining nonlinear OCE specifications generally lead to nonlinear portfolio problems, motivating the certified approximation framework developed next.

\section{Approximation and Stability of Recovery OCE Optimization}
\label{sec:approx}

We develop a certified piecewise-linear approximation framework for nonlinear Recovery OCEs under the finite recovery specification \eqref{eq:finiteRecoverySpecification}. OCE localization reduces utility approximation to compact residual intervals, allowing approximation errors to be transferred to the Recovery OCE objective and, ultimately, to the associated portfolio problem. The piecewise-linear structure additionally yields a finite-dimensional linear programming formulation.

\subsection{Piecewise-Linear Approximation and Linear Programming}
General OCE risk measures provide considerable flexibility through the choice of the underlying utility function. To obtain tractable finite-scenario formulations, we approximate the regime-wise utilities by concave piecewise-linear functions that preserve the relevant OCE structure.

For each recovery regime \(j=1,\ldots,l+1\), we replace \(u_{\alpha_j}\) by  a concave piecewise-linear approximation
\(\widehat u_{\alpha_j}\) of the form
\[
    \widehat u_{\alpha_j}(x)
    =
    \min_{i=1,\ldots,n_j}
    \left\{
        a_i^{(j)}x+b_i^{(j)}
    \right\},
    \qquad j=1,\ldots,l+1.
\]
We require \(\widehat u_{\alpha_j}\) to be nondecreasing and normalized by \(\widehat u_{\alpha_j}(0)=0\) and \(1\in\partial\widehat u_{\alpha_j}(0)\), so that \(\widehat u_{\alpha_j}\in\mathcal U_0\). We denote the resulting finite approximating family by \(\widehat{\mathfrak u}:=(\widehat u_{\alpha_j})_{j=1}^{l+1}\). The approximating utilities are not required to be globally ordered; it suffices that their ordering is preserved on the residual intervals relevant to the localized OCE problems below. Throughout this section, we use the notation
\[
\operatorname{Rec}\rho_{\widehat{u},\gamma}(X,Y)
:=
\max_{j=1,\ldots,l+1}
\rho_{\widehat{u}_{\alpha_j}}
\left(X+(1-r_j)Y\right)
\]
for the corresponding finite-regime approximation.

\begin{remark}[Structure-preserving interpolation]
\label{rem:admissible-pwl-approximations}
Such approximations can be constructed by linear interpolation on a common
grid \(G=\{x_0<\cdots<x_N\}\) containing \(0\) as an interior grid point,
followed by affine continuation of the outermost pieces. Concavity and
monotonicity of \(u_{\alpha_j}\) imply that the resulting interpolant is
concave and nondecreasing, while \(u_{\alpha_j}(0)=0\) and
\(1\in\partial u_{\alpha_j}(0)\) imply
\(\widehat u_{\alpha_j}(0)=0\) and
\(1\in\partial\widehat u_{\alpha_j}(0)\). Hence
\(\widehat u_{\alpha_j}\in\mathcal U_0\).

Moreover, common-grid interpolation preserves the ordering of the utility
specifications on the interval \([x_0,x_N]\): if
\(u_{\alpha_1}\leq\cdots\leq u_{\alpha_{l+1}}\), then
\[
\widehat u_{\alpha_1}(x)
\leq\cdots\leq
\widehat u_{\alpha_{l+1}}(x),
\qquad x\in[x_0,x_N].
\]
Concavity further implies
\(\widehat u_{\alpha_j}\leq u_{\alpha_j}\) on \([x_0,x_N]\) for every
\(j\). It therefore suffices below to choose the interpolation interval so
that it contains the residual ranges relevant to the localized OCE problems.
\end{remark}

For the finite-scenario portfolio problem \eqref{eq:finiteRecOCEPortfolioProblem}, the piecewise-linear utility approximations yield a finite-dimensional linear programming formulation. Replacing \(u_{\alpha_j}\) by \(\widehat u_{\alpha_j}\) yields
\begin{equation}
\label{eq:pwl-recoce-portfolio}
\inf_{\mathbf w\in\Delta_K(\mu)}
\max_{j=1,\ldots,l+1}
\rho_{\widehat u_{\alpha_j}}
\bigl(\mathbf A(\mathbf w)-r_j\mathbf L\bigr).
\end{equation}
Continuity of the objective and compactness of \(\Delta_K(\mu)\) ensure existence of an optimal solution.

\begin{proposition}
\label{prop:pwl-portfolio-lp}
The approximating portfolio problem
\eqref{eq:pwl-recoce-portfolio} is equivalent to the linear program
\[
\begin{aligned}
    \min_{\mathbf w,\mathbf m,\mathbf v,T}\quad
        & T
    \\
    \text{\rm s.t.}\quad
        &
        -m_j+\sum_{s=1}^S p_s v_s^{(j)}
        \leq T,
        && j=1,\ldots,l+1,
    \\
        &
        v_s^{(j)}
        \geq
        -a_i^{(j)}
        \left(
            A_s(\mathbf w)-r_jL_s-m_j
        \right)
        -b_i^{(j)},
        &&
        \begin{array}{l}
        i=1,\ldots,n_j,\\
        j=1,\ldots,l+1,\\
        s=1,\ldots,S,
        \end{array}
    \\
        &
        \sum_{s=1}^S p_s\,\mathbf w^\top\mathbf R_s
        \geq\mu,
    \\
        &
        \sum_{k=1}^K w_k=1,
        \qquad
        w_k\geq0,\quad k=1,\ldots,K,
    \\
        &
        \mathbf m\in\mathbb R^{l+1},
        \qquad
        \mathbf v\in\mathbb R^{S\times(l+1)},
        \qquad
        T\in\mathbb R.
\end{aligned}
\]
\end{proposition}

\begin{proof}
For each regime \(j\), the OCE definition, \( -\widehat u_{\alpha_j}(x) = \max_{i=1,\ldots,n_j} \{-a_i^{(j)}x-b_i^{(j)}\} \), and the \Cref{lemma:oce-sup-attained}  yield
\[
\rho_{\widehat u_{\alpha_j}}
\bigl(\mathbf A(\mathbf w)-r_j\mathbf L\bigr)
=
\min_{m_j\in\mathbb R}
\left\{
-m_j+
\sum_{s=1}^S p_s
\max_{i=1,\ldots,n_j}
\left[
-a_i^{(j)}
\bigl(A_s(\mathbf w)-r_jL_s-m_j\bigr)
-b_i^{(j)}
\right]
\right\}.
\]
Introducing \(v_s^{(j)}\) as epigraph variables for the scenario-wise maxima and \(T\) as an epigraph variable for the maximum across recovery regimes gives the stated formulation. For fixed $\mathbf w$ and $\mathbf m$, replacing each $v_s^{(j)}$ by its smallest feasible value preserves feasibility and leaves the objective unchanged. Hence the resulting epigraph reformulation is exact. Since \(A_s(\mathbf w)=b(1+\mathbf w^\top\mathbf R_s)\) is affine in \(\mathbf w\), all constraints are affine.
\end{proof}

\subsection{Approximation Guarantees for Recovery OCE Portfolio Optimization}
\label{sec:approximation-guarantees}

The preceding reformulation establishes the computational tractability of the piecewise-linear approximation. We now quantify how approximation errors in the utility functions propagate to Recovery OCE risk measures and the associated portfolio problems.  We first derive bounds and error estimates for fixed positions and then extend them uniformly over the feasible portfolio set to obtain convergence of optimal values and stability of optimal solutions.

\begin{lemma}[Sandwich bounds]
\label{lemma:recoce-sandwich}
Let \(\gamma\) be the finite recovery specification introduced in \eqref{eq:finiteRecoverySpecification}, and let \(\mathfrak u^-=(u^-_{\alpha_j})_{j=1}^{l+1}\), \(\mathfrak u=(u_{\alpha_j})_{j=1}^{l+1}\), and \(\mathfrak u^+=(u^+_{\alpha_j})_{j=1}^{l+1}\) be utility specifications in \(\mathcal U_0\) such that
\[
    u^-_{\alpha_j}(x)
    \leq u_{\alpha_j}(x)
    \leq u^+_{\alpha_j}(x),
    \qquad
    x\in\mathbb R,\quad j=1,\ldots,l+1.
\]
If the associated OCEs are finite-valued on \(L^\infty\), then, for every \(X\in L^\infty\) and \(Y\in L^\infty_+\),
\[
    \operatorname{Rec}\rho_{\mathfrak u^+,\gamma}(X,Y)
    \leq
    \operatorname{Rec}\rho_{\mathfrak u,\gamma}(X,Y)
    \leq
    \operatorname{Rec}\rho_{\mathfrak u^-,\gamma}(X,Y).
\]
\end{lemma}

\begin{proof}
The ordering of the utility functions implies
\(\rho_{u^+_{\alpha_j}}(Z)\leq\rho_{u_{\alpha_j}}(Z)
\leq\rho_{u^-_{\alpha_j}}(Z)\) for every \(Z\in L^\infty\).
Applying this to \(Z=X+(1-r_j)Y\) and taking the maximum over
\(j=1,\ldots,l+1\) yields the claim.
\end{proof}

Whenever these one-sided approximations hold globally, secant interpolants provide lower utility approximations and hence upper bounds on the Recovery OCE, whereas the pointwise minimum of supporting affine functions provides upper utility approximations and hence lower bounds.

For a fixed position, global approximation is not necessary. By localization of the OCE optimization, it suffices to control the utilities on suitable compact residual ranges.  Let \(X\in L^\infty\), \(Y\in L^\infty_+\), and, for \(j=1,\ldots,l+1\), define
\[
    Z_j:=X+(1-r_j)Y,\qquad
    z_j^-:=\operatorname{ess\,inf}Z_j,\qquad
    z_j^+:=\operatorname{ess\,sup}Z_j.
\]
By \Cref{lemma:oce-sup-attained}, the OCE optimization at recovery level \(r_j\) can be restricted to \(m\in[z_j^-,z_j^+]\). Hence, all relevant residuals
\(Z_j-m\) lie in
\[
    I_j:=[z_j^- - z_j^+,\,z_j^+-z_j^-].
\]

\begin{lemma}[Error propagation]
\label{lemma:recoce-error-propagation}
Let \(\mathfrak u=(u_{\alpha_j})_{j=1}^{l+1}\) and
\(\widehat{\mathfrak u}=(\widehat u_{\alpha_j})_{j=1}^{l+1}\)
be utility specifications in \(\mathcal U_0\). Suppose that
\(u_{\alpha_j}\) and \(\widehat u_{\alpha_j}\) are finite on \(I_j\), and set
\[
    \varepsilon_j
    :=
    \sup_{x\in I_j}
    |u_{\alpha_j}(x)-\widehat u_{\alpha_j}(x)|,
    \qquad j=1,\ldots,l+1.
\]
Then
\[
    \left|
        \operatorname{Rec}\rho_{\mathfrak u,\gamma}(X,Y)
        -
        \operatorname{Rec}\rho_{\widehat{\mathfrak u},\gamma}(X,Y)
    \right|
    \leq
    \max_{j=1,\ldots,l+1}\varepsilon_j.
\]
\end{lemma}

\begin{proof}
By \Cref{lemma:oce-sup-attained}, both OCE optimizations in regime \(j\) can be restricted to \(m\in[z_j^-,z_j^+]\), where \(Z_j-m\in I_j\). Therefore,
\[
    \left|
        \rho_{u_{\alpha_j}}(Z_j)
        -
        \rho_{\widehat u_{\alpha_j}}(Z_j)
    \right|
    \leq
    \sup_{m\in[z_j^-,z_j^+]}
    \left|
        \mathbb E[u_{\alpha_j}(Z_j-m)]
        -
        \mathbb E[\widehat u_{\alpha_j}(Z_j-m)]
    \right|
    \leq\varepsilon_j.
\]
Taking the maximum over the recovery regimes and using
\(
|\max_j a_j-\max_j b_j|\leq\max_j|a_j-b_j|
\)
yields the claim.
\end{proof}

Thus, approximation rates for the utilities on the relevant residual ranges transfer directly, without loss of order, to the Recovery OCE.

\begin{corollary}[Approximation rates]
\label{cor:recoce-approximation-rates}
Suppose that the utilities \(u_{\alpha_j}\) are approximated by the piecewise-linear secant interpolants from \Cref{rem:admissible-pwl-approximations} on a common grid \(x_0<\cdots<x_N\) with mesh size \(h:=\max_{k=1,\ldots,N}(x_k-x_{k-1})\), whose interpolation interval \(I=[x_0,x_N]\) contains all residual ranges \(I_j\).

If \(u_{\alpha_j}\in C^2(I)\) for all \(j\) and \(\max_j\|u_{\alpha_j}''\|_{\infty,I}\leq M\), then
\[
    \left|
        \operatorname{Rec}\rho_{\mathfrak u,\gamma}(X,Y)
        -
        \operatorname{Rec}\rho_{\widehat{\mathfrak u},\gamma}(X,Y)
    \right|
    \leq
    \frac{M}{8}h^2.
\]
If, instead, the utilities are Lipschitz continuous on \(I\) with Lipschitz constants uniformly bounded by \(L\), then
\[
    \left|
        \operatorname{Rec}\rho_{\mathfrak u,\gamma}(X,Y)
        -
        \operatorname{Rec}\rho_{\widehat{\mathfrak u},\gamma}(X,Y)
    \right|
    \leq
    \frac{L}{2}h.
\]
Thus, the Recovery OCE approximation converges with order \(O(h^2)\) and \(O(h)\), respectively.
\end{corollary}

\begin{proof}
For the \(C^2\) case, the standard error bound for linear interpolation \cite{atkinson2008introduction} gives
\[
    \|u_{\alpha_j}-\widehat u_{\alpha_j}\|_{\infty,I}
    \leq
    \frac{M}{8}h^2,
    \qquad j=1,\ldots,l+1.
\]
For the Lipschitz case, if \(x=(1-t)x_{k-1}+tx_k\) on a grid interval,
then
\[
\begin{aligned}
    |u_{\alpha_j}(x)-\widehat u_{\alpha_j}(x)|
    &\leq
    L(1-t)(x-x_{k-1})+Lt(x_k-x) \\
    &=2Lt(1-t)(x_k-x_{k-1})
    \leq \frac{L}{2}h.
\end{aligned}
\]
Both claims now follow from \Cref{lemma:recoce-error-propagation}.
\end{proof}

\begin{remark}[Adaptive refinement]
\label{rem:adaptive-refinement}
The sandwich bounds of \Cref{lemma:recoce-sandwich} suggest a simple adaptive refinement strategy. Suppose that, for each recovery regime \(j\), lower and upper approximations yield bounds \(\underline R_j\leq\rho_{u_{\alpha_j}}(Z_j)\leq\overline R_j\). Combining these gives \(\max_j\underline R_j \leq\operatorname{Rec}\rho_{\mathfrak u,\gamma}(X,Y) \leq\max_j\overline R_j\). In particular, any regime \(j\) satisfying
\[
    \overline R_j<\max_k\underline R_k
\]
cannot attain the maximal regime-wise risk and may therefore be excluded from further refinement. The approximations need only be refined for regimes whose certified intervals still overlap the current lower bound for the maximum, thereby tightening the bounds where they can affect the Recovery OCE. This rule is position-specific and should not be confused with the universal regime redundancy studied in \Cref{sec:redundancy}.
\end{remark}

We now extend the preceding error bounds uniformly over the entire feasible portfolio set. Let \(\widehat{\mathfrak u}_n =(\widehat u_{\alpha_j,n})_{j=1}^{l+1}\), \(n\in\mathbb N\), be a sequence of piecewise-linear approximations of \(\mathfrak u=(u_{\alpha_j})_{j=1}^{l+1}\), and define
\[
    F(\mathbf w)
    :=
    \operatorname{Rec}\rho_{\mathfrak u,\gamma}
    (\mathbf A(\mathbf w)-\mathbf L,\mathbf L),
    \qquad
    F_n(\mathbf w)
    :=
    \operatorname{Rec}\rho_{\widehat{\mathfrak u}_n,\gamma}
    (\mathbf A(\mathbf w)-\mathbf L,\mathbf L),
\]
for \(\mathbf w\in\Delta_K(\mu)\).

For each recovery regime \(j\), let \(\mathbf Z_j(\mathbf w):=\mathbf A(\mathbf w)-r_j\mathbf L\), \(\mathbf w\in\Delta_K(\mu)\), denote the regime-wise residual vector in \eqref{eq:finiteRecOCEPortfolioProblem}, with scenario components \(Z_{j,s}(\mathbf w)\), \(s=1,\ldots,S\). Set
\[
    \underline z_j
    :=
    \min_{\mathbf w\in\Delta_K(\mu)}
    \min_{s=1,\ldots,S} Z_{j,s}(\mathbf w),
    \qquad
    \overline z_j
    :=
    \max_{\mathbf w\in\Delta_K(\mu)}
    \max_{s=1,\ldots,S} Z_{j,s}(\mathbf w).
\]
Since \(\Delta_K(\mu)\) is nonempty and compact and \(Z_{j,s}\) is continuous in \(\mathbf w\), these extrema are finite and attained. By \Cref{lemma:oce-sup-attained}, all residual values arising in the localized OCE problems are contained, uniformly over \(\mathbf w\in\Delta_K(\mu)\), in the compact interval
\[
    I_j^{\mathrm{port}}
    :=
    [\underline z_j-\overline z_j,\,
     \overline z_j-\underline z_j].
\]

\begin{theorem}[Stability under utility approximation]
\label{theorem:portfolio-stability}
Suppose that \(u_{\alpha_j}\) and \(\widehat u_{\alpha_j,n}\) are finite on
\(I_j^{\mathrm{port}}\) for all \(j\) and \(n\), and set
\[
    \varepsilon_n
    :=
    \max_{j=1,\ldots,l+1}
    \sup_{x\in I_j^{\mathrm{port}}}
    |u_{\alpha_j}(x)-\widehat u_{\alpha_j,n}(x)|.
\]
If \(\varepsilon_n\to0\), then the following statements hold.

\begin{enumerate}
    \item[(i)] \emph{Uniform convergence of the portfolio objectives.}
    For every \(n\in\mathbb N\), the uniform error satisfies
    \(\sup_{\mathbf w\in\Delta_K(\mu)}
    |F_n(\mathbf w)-F(\mathbf w)|\leq\varepsilon_n\).
    In particular, \(F_n\to F\) uniformly on \(\Delta_K(\mu)\).

    \item[(ii)] \emph{Stability of the optimal values.}
    The functions \(F\) and \(F_n\) attain their minima on
    \(\Delta_K(\mu)\). Setting
    \(V:=\min_{\mathbf w\in\Delta_K(\mu)}F(\mathbf w)\) and
    \(V_n:=\min_{\mathbf w\in\Delta_K(\mu)}F_n(\mathbf w)\), we have
    \(|V_n-V|\leq\varepsilon_n\), and hence \(V_n\to V\).

    \item[(iii)] \emph{Stability of optimal solutions.}
    If
    \(\mathbf w_n\in
    \operatorname*{arg\,min}_{\mathbf w\in\Delta_K(\mu)}F_n(\mathbf w)\),
    then every accumulation point of \((\mathbf w_n)_{n\in\mathbb N}\)
    belongs to
    \(\operatorname*{arg\,min}_{\mathbf w\in\Delta_K(\mu)}F(\mathbf w)\).
    If the exact problem has a unique minimizer \(\mathbf w^\star\), then
    \(\mathbf w_n\to\mathbf w^\star\).
\end{enumerate}
\end{theorem}

\begin{proof}
    The proof is given in \Cref{app:portfolio-stability}.
\end{proof}

In particular, for the secant approximations of \Cref{cor:recoce-approximation-rates} on a common compact interval \(I\) containing all \(I_j^{\mathrm{port}}\), with mesh sizes \(h_n\to0\), the optimal values inherit the same convergence rates: \(|V_n-V|\leq Mh_n^2/8\) under the \(C^2\)-conditions and \(|V_n-V|\leq Lh_n/2\) under the Lipschitz conditions.

\begin{remark}[Extension to distributional ambiguity]
The preceding approximation arguments extend naturally to a distributionally robust setting; see, e.g., \cite{zhu2009worst,delage2010distributionally,mohajerin2018data} for established approaches to distributionally robust risk and portfolio optimization. Let \(\mathcal P\) be a nonempty ambiguity set of probability measures and define
\[
    \operatorname{Rec}\rho_{\mathfrak u,\gamma}^{\mathcal P}(X,Y)
    :=
    \sup_{P\in\mathcal P}
    \operatorname{Rec}\rho_{\mathfrak u,\gamma}^{P}(X,Y).
\]
If the utility-approximation error bounds hold uniformly over
\(P\in\mathcal P\), then
\[
    \left|
        \operatorname{Rec}\rho_{\mathfrak u,\gamma}^{\mathcal P}(X,Y)
        -
        \operatorname{Rec}\rho_{\widehat{\mathfrak u},\gamma}^{\mathcal P}(X,Y)
    \right|
    \leq
    \max_j\varepsilon_j.
\]
Thus, taking the worst case over the ambiguity set does not amplify the
utility-approximation error. Consequently, the preceding uniform convergence
and optimal-value bounds extend whenever the corresponding approximation
estimates hold uniformly over both the ambiguity set and the feasible
portfolios. Stability of optimal solutions then follows under the same
compactness and continuity conditions as above.

Such uniform control is immediate when the utility approximation is valid on
a common compact residual range containing all relevant residuals under every
\(P\in\mathcal P\). If no such common compact range exists, additional global
approximation or suitable growth and integrability conditions are required.
\end{remark}

Taken together, the results of this section show that piecewise-linear approximation provides a tractable and certified framework for Recovery OCE portfolio optimization. Approximation errors in the utility functions can be controlled on compact residual ranges and transferred directly to the Recovery OCE objective, yielding explicit error bounds and convergence rates as well as stability guarantees for the associated portfolio problems. In particular, approximate optimal values converge to the exact optimal value, while accumulation points of approximate optimal solutions solve the exact problem; if the exact optimizer is unique, the approximate optimal solutions converge to it.

\section{Numerical Study: Recovery Requirements in Asset--Liability
Portfolio Optimization}

This section illustrates the theoretical and computational implications of Recovery OCE risk measures in a finite-scenario asset--liability portfolio problem. We first compare standard AV@R with RecAV@R to assess the impact of recovery requirements on optimal portfolio allocations and to identify the recovery regimes that are active along the efficient frontier, complementing the structural analysis of \Cref{sec:dual}. For the entropic specification, we then examine the accuracy of the piecewise-linear approximation and the convergence and stability of the resulting portfolio problem as the approximation grid is refined, illustrating the results of \Cref{sec:approx}.

\subsection{Model Setup}

The numerical study considers an asset--liability portfolio problem with a diversified investment universe comprising the iShares MSCI World ETF (URTH), the Vanguard Real Estate ETF (VNQ), the iShares MSCI Global Min Vol Factor ETF (ACWV), the iShares 7--10 Year Treasury Bond ETF (IEF), and the iShares 1--3 Year Treasury Bond ETF (SHY). The assets provide a parsimonious set of equity, real-estate, and fixed-income exposures with distinct risk characteristics.

Daily adjusted closing prices are obtained from Yahoo Finance and converted to monthly log returns using the last available adjusted closing price in each calendar month. The resulting sample covers February 2012 through December 2021. Adjusted prices are used to account for distributions and corporate actions reflected in the Yahoo Finance data.

We generate \(1{,}000\) one-year asset-return scenarios by sampling twelve monthly return vectors independently with replacement from the historical sample. Within each sampled month, the same historical observation is used across all assets, thereby preserving the contemporaneous cross-sectional dependence of asset returns. The twelve sampled monthly log returns are summed and subsequently transformed into annual simple returns. This yields an empirical one-year return distribution while abstracting from serial dependence in monthly returns. The bootstrap scenarios were generated using MATLAB's pseudorandom number generator with random seed~1. All computations were performed in MATLAB R2025b.

Throughout the numerical study, initial wealth is normalized to $b=1$, so that all asset and liability values are expressed relative to initial wealth. To introduce dependence between asset returns and liabilities, the liability is linked to the simulated annual performance of the broad equity market.
Let \(R^{\mathrm{World}}_s\) denote the one-year simple return of URTH in scenario \(s\), and let \(q_{0.05}^{\mathrm{World}}\) denote its empirical \(5\%\)-quantile across the simulated scenarios. The scenario-dependent liability is defined by
\[
    L_s
    =
    \begin{cases}
        0.9,
        & R^{\mathrm{World}}_s
          \leq q_{0.05}^{\mathrm{World}},\\
        0.6,
        & R^{\mathrm{World}}_s
          > q_{0.05}^{\mathrm{World}}.
    \end{cases}
\]
Thus, the liability increases by \(50\%\) in scenarios associated with the lowest simulated annual returns of the broad equity-market proxy. This stylized specification introduces adverse dependence between asset-market conditions and liabilities while keeping the asset--liability structure transparent.

The parameter choices used throughout the numerical study are illustrative rather than calibrated and are not intended to represent a particular regulatory or empirical specification. They are chosen to produce distinct risk assessments across the recovery regimes, allowing the interaction between recovery requirements, risk preferences, and optimal portfolio allocations to be examined. In applications, the recovery levels and corresponding risk-preference parameters would have to be selected according to the decision maker’s requirements and preferences.

\subsection{Portfolio Effects and Active Recovery Regimes}
\label{subsec:numerical_avar}

We first examine the portfolio implications of recovery requirements by comparing standard AV@R with RecAV@R under identical asset--liability scenarios and portfolio constraints. Particular attention is given to the recovery regimes that become active along the RecAV@R-efficient frontier and to the associated changes in optimal portfolio allocations.

For the AV@R specification, we use $\operatorname{AV@R}_{0.10}$ as the benchmark, while the corresponding RecAV@R specification comprises the three recovery regimes
\[
    (r_j,\alpha_j)
    \in
    \bigl\{
        (0.80,0.02),
        (0.90,0.06),
        (1.00,0.10)
    \bigr\}.
\]
Accordingly, the Recovery OCE portfolio problem \eqref{eq:finiteRecOCEPortfolioProblem} specializes to
\[
    \inf_{\mathbf w\in\Delta_K(\mu)}
    \max_{j=1,2,3}
    \operatorname{AV@R}_{\alpha_j}
    \bigl(\mathbf A(\mathbf w)-r_j\mathbf L\bigr).
\]
The full-recovery regime $(r_3,\alpha_3)=(1,0.10)$ coincides with the standard AV@R benchmark. The recovery specification therefore retains the benchmark requirement while imposing additional requirements at the lower recovery levels $r=0.9$ and $r=0.8$.

\Cref{fig:avar_frontiers} compares the RecAV@R-efficient frontier with the standard AV@R frontier and the frontiers obtained by optimizing the individual recovery regimes separately. While the latter provide regime-specific benchmarks, the RecAV@R frontier results from the joint minimax problem, requiring a single portfolio to account for all recovery regimes simultaneously. In the present numerical setting, the RecAV@R frontier coincides, up to numerical precision, with the upper envelope of the regime-specific efficient frontiers. As the target return increases, different regime-specific frontiers form this upper envelope, indicating that the recovery requirement relevant for the joint problem changes along the efficient frontier. The active regimes at the RecAV@R-optimal portfolios are examined directly below.

To identify the recovery requirements that actually determine the RecAV@R objective at the optimal portfolio, define the regime-specific risk at
portfolio $\mathbf w$ by
\[
    R_j(\mathbf w)
    :=
    \operatorname{AV@R}_{\alpha_j}
    \bigl(\mathbf A(\mathbf w)-r_j\mathbf L\bigr),
    \qquad j=1,2,3.
\]
For a RecAV@R-optimal portfolio $\mathbf w^\star(\mu)$, the corresponding active set is
\[
    \mathcal J\bigl(\mathbf w^\star(\mu)\bigr)
    :=
    \left\{
        j:
        R_j\bigl(\mathbf w^\star(\mu)\bigr)
        =
        \max_{k=1,2,3}
        R_k\bigl(\mathbf w^\star(\mu)\bigr)
    \right\}.
\]
The active set need not be a singleton, as several recovery regimes may simultaneously attain the maximal risk at the optimal portfolio. In the numerical implementation, regime activity is identified up to a tolerance $\varepsilon=10^{-7}$.

\Cref{fig:avar_allocations} illustrates the associated portfolio effects and active recovery regimes. Panel~(a) shows the optimal allocations under the standard $\operatorname{AV@R}_{0.10}$ benchmark, while panel~(b) shows the RecAV@R-optimal allocations for the same target returns. The horizontal segments above panel~(b) indicate the recovery regimes contained in $\mathcal J(\mathbf w^\star(\mu))$. Overlapping segments thus identify target-return ranges in which multiple recovery regimes are simultaneously active.

At lower target returns, the full-recovery regime $r=1$ is active, with the intermediate recovery regime $r=0.9$ becoming simultaneously active over part of this range. The RecAV@R allocation initially remains close to the standard AV@R allocation. At higher target returns, the lower-recovery regime $r=0.8$ becomes active. The overlapping activity regions show that several recovery requirements may jointly determine the objective at the optimal portfolio. Overall, the comparison illustrates that additional recovery requirements can materially affect optimal portfolio allocations even though the standard AV@R requirement is retained as the full-recovery regime.

It is important to distinguish the regime activity observed along the efficient frontier from the global notion of effectiveness introduced in \Cref{sec:dual}. Activity is portfolio-specific and, in the present numerical setting, identifies the recovery regimes attaining the RecAV@R maximum at the optimal portfolios generated by the specified asset--liability scenarios. Effectiveness, by contrast, is a global property of the recovery specification for the underlying probability model and is defined over the entire domain of financial positions. Hence, the active-set patterns in \Cref{fig:avar_allocations} characterize regime activity for the particular numerical setting considered here, but do not, by themselves, establish global effectiveness or redundancy.

\begin{figure}[t]
    \centering

    \begin{minipage}[t]{0.48\textwidth}
        \centering
        \includegraphics[width=\linewidth]
        {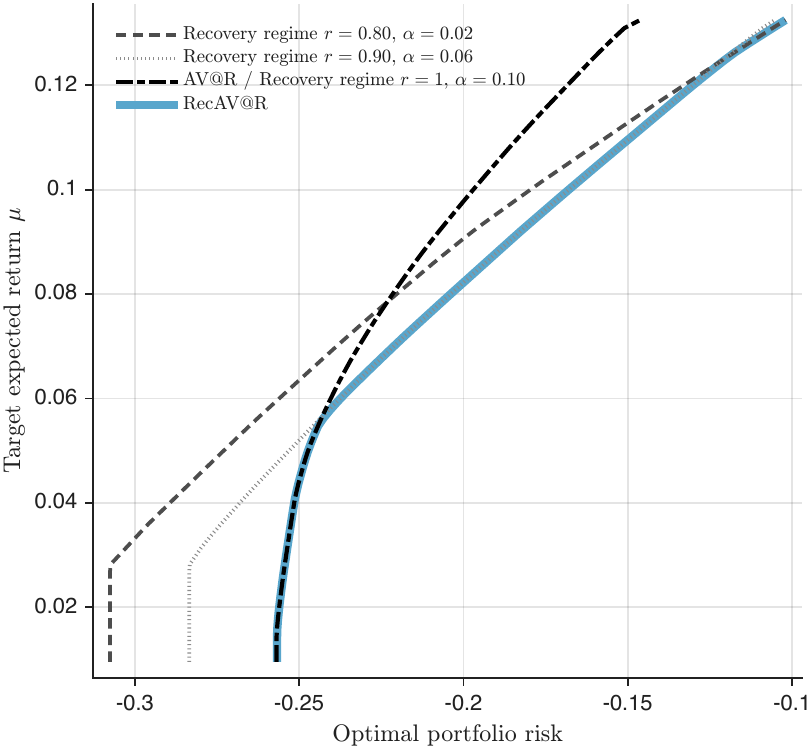}
        \caption{Efficient frontiers under the regime-specific
        AV@R requirements and the joint RecAV@R specification.}
        \label{fig:avar_frontiers}
    \end{minipage}
    \hfill
    \begin{minipage}[t]{0.48\textwidth}
        \centering
        \includegraphics[width=\linewidth]
        {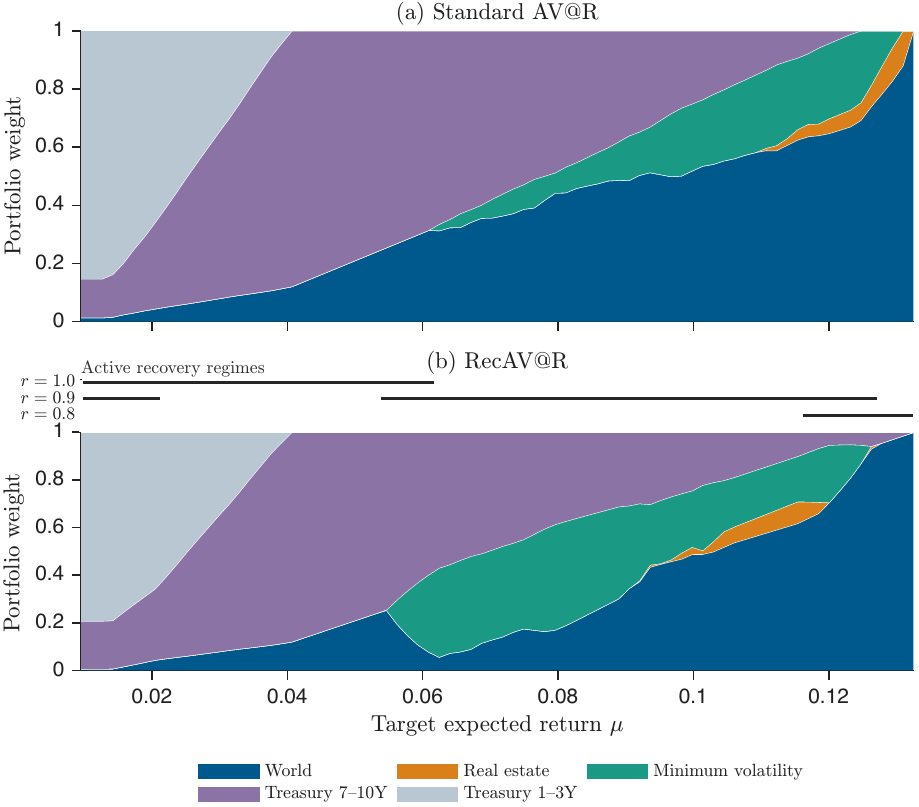}
        \caption{Optimal portfolio allocations under standard
        AV@R and RecAV@R, with active recovery regimes shown for RecAV@R.}
        \label{fig:avar_allocations}
    \end{minipage}

\end{figure}

\subsection{Approximation Accuracy and Convergence}
\label{subsec:numerical_approximation}

Having illustrated the portfolio and regime effects of recovery requirements, we next examine the accuracy of the piecewise-linear utility approximation and the convergence and stability of the resulting Recovery OCE portfolio problem. To this end, we consider a Recovery OCE generated by the exponential utility family  $u_\theta(x)=\theta\left(1-e^{-x/\theta}\right)$, $\theta>0$, introduced in \Cref{ex:standard-oce-risk-measures}, whose associated OCE is the entropic risk measure. The entropic specification serves as a smooth nonlinear OCE benchmark rather than as a calibrated model of regulatory or institutional risk measurement. The finite recovery specification consists of the four regimes
\[
    (r_j,\theta_j)
    \in
    \left\{
        \left(0.80,\tfrac{1}{21}\right),
        \left(0.85,\tfrac{1}{16}\right),
        \left(0.90,\tfrac{1}{13}\right),
        \left(1.00,\tfrac{1}{8}\right)
    \right\},
\]
and the target expected return is set to \(\mu=0.10\). Before solving the portfolio problem, we apply the structural screening criterion of \Cref{proposition:structuralScreening}. The regime $(0.85,\tfrac{1}{16})$ is universally redundant, since \(0.85=\tfrac12(0.80)+\tfrac12(0.90)\) and \(\tfrac1{16}>\tfrac12(\tfrac1{21}+\tfrac1{13})\). It can therefore be removed without changing the Recovery OCE for any asset--liability position. The remaining three regimes are not eliminated by the structural screening criterion and are therefore retained in the subsequent optimization. 

We approximate the regime-specific utility functions using the piecewise-linear constructions developed in \Cref{sec:approx}. For each recovery regime \(j\), the approximation is constructed on the portfolio residual interval
\[
    I_j^{\mathrm{port}}
    =
    [\underline z_j-\overline z_j,\,
     \overline z_j-\underline z_j].
\]
By construction, \(I_j^{\mathrm{port}}\) contains all residual values arising in the localized regime-wise OCE problems, uniformly over all feasible portfolios.

Since \(I_j^{\mathrm{port}}\) is symmetric around zero, we use an odd number \(N_{\mathrm{grid}}\) of equally spaced support points, with zero as the central support point. The secant construction uses the \(N_{\mathrm{grid}}-1\) chords between adjacent support points and yields a lower approximation \(u_{\theta_j,N_{\mathrm{grid}}}^{\mathrm{sec}}\) of \(u_{\theta_j}\) on \(I_j^{\mathrm{port}}\), whereas the lower envelope of the supporting tangents at the same \(N_{\mathrm{grid}}\) points yields an upper approximation \(u_{\theta_j,N_{\mathrm{grid}}}^{\mathrm{tan}}\). Substituting the piecewise-linear utilities into the portfolio problem yields the linear programming formulations developed in \Cref{prop:pwl-portfolio-lp}. The nonlinear Rec-entropic benchmark is computed using MATLAB's \texttt{fmincon} with the SQP algorithm and optimality and constraint tolerances of \(10^{-10}\). The piecewise-linear approximation problems are solved using MATLAB's \texttt{linprog} with the dual-simplex algorithm and its default solver tolerances. By localization of the regime-wise OCE problems and the regime-wise ordering of the utility approximations on the relevant residual intervals, their optimal values provide the certified bounds
\[
    V_{N_{\mathrm{grid}}}^{\mathrm{tan}}
    \leq
    V
    \leq
    V_{N_{\mathrm{grid}}}^{\mathrm{sec}},
\]
where \(V\) denotes the optimal value of the original nonlinear portfolio optimization problem and \(V_{N_{\mathrm{grid}}}^{\mathrm{tan}}\) and \(V_{N_{\mathrm{grid}}}^{\mathrm{sec}}\) denote the optimal values of the tangent- and secant-based linear programs. For the nonlinear Rec-entropic portfolio problem, the numerical benchmark is $V=-0.33581$. \Cref{fig:entropic_approximation} shows how the certified interval $[V_{N_{\mathrm{grid}}}^{\mathrm{tan}},V_{N_{\mathrm{grid}}}^{\mathrm{sec}}]$ contracts around this benchmark as the number of support points \(N_{\mathrm{grid}}\) increases. At the finest grid considered, with \(N_{\mathrm{grid}}=35\), the bounds are $V_{35}^{\mathrm{tan}}=-0.33867$ and $V_{35}^{\mathrm{sec}}=-0.32925,$ corresponding to an interval width of approximately \(0.00942\). As guaranteed by the one-sided approximation bounds, the nonlinear benchmark remains within the certified interval throughout the refinement sequence. The interval contracts as the grid is refined, with both bounds approaching the nonlinear benchmark. Using the finer grids, a log–log regression of the secant approximation error against the maximum mesh width yields an empirical convergence order close to two, consistent with the quadratic $O(h^2)$ rate derived above for the optimal value. The tangent approximation likewise converges toward the nonlinear benchmark in the numerical experiment, although its empirical rate is less stable across the considered grid resolutions. The corresponding approximate optimizers are also nearly optimal under the original nonlinear objective: at \(N_{\mathrm{grid}}=35\), the secant- and tangent-based solutions have optimality gaps of only \(1.64\times10^{-4}\) and \(2.74\times10^{-5}\), respectively.

\begin{figure}[t]
    \centering
    \includegraphics[width=0.85\textwidth]
        {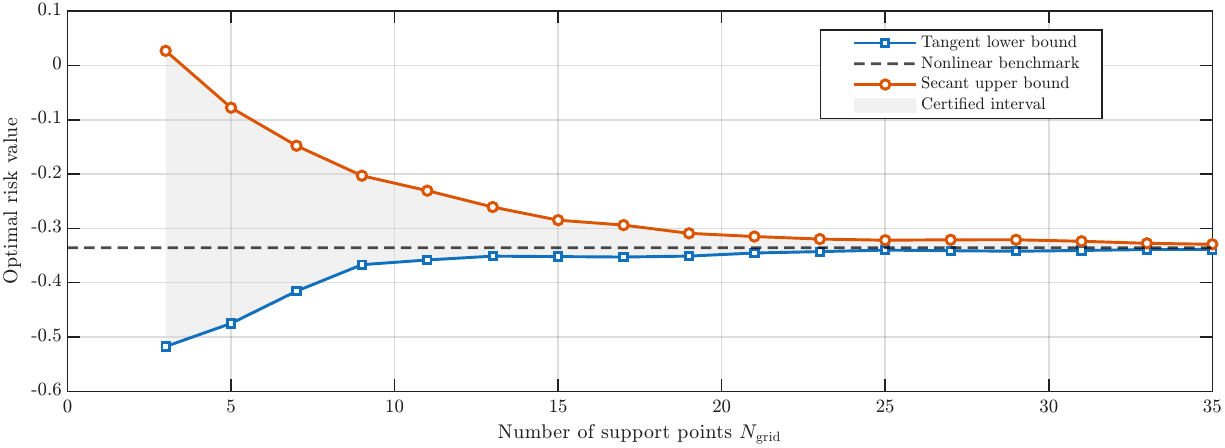}
    \caption{Certified approximation of the optimal Rec-entropic risk value. Tangent and secant utility approximations provide lower and upper bounds, respectively, while the shaded region denotes the resulting certified interval.}
    \label{fig:entropic_approximation}
\end{figure}

Finally, the regime-wise lower and upper bounds \(\underline R_{j,N_{\mathrm{grid}}}\) and \(\overline R_{j,N_{\mathrm{grid}}}\), respectively, illustrate the screening principle underlying \Cref{rem:adaptive-refinement}. At the nonlinear benchmark portfolio, the intermediate recovery regime \(r=0.9\) is active. At \(N_{\mathrm{grid}}=25\), the full-recovery regime satisfies
\[
    \overline R_{r=1,25}
    =
    -0.34185
    <
    -0.33976
    =
    \max_j \underline R_{j,25},
\]
and therefore cannot attain the maximal regime-specific risk at this portfolio. It may consequently be excluded from further refinement at this position, illustrating how certified approximation intervals can be used to avoid unnecessary regime-wise refinement. This position-specific refinement screening is distinct from the structural screening applied above: the latter removes a universally redundant recovery regime before optimization, whereas the former uses certified approximation bounds to avoid further refinement of a regime at a particular portfolio.

Overall, the numerical results illustrate both the portfolio implications and the computational tractability of Recovery OCE portfolio optimization. Recovery requirements affect portfolio composition and regime activity, while the piecewise-linear approximations provide certified bounds that tighten under refinement and yield approximate optimizers that are nearly optimal under the exact objective.

\section{Conclusions}
\label{sec:conclusions}

This paper develops a structural and computational theory for Recovery OCEs with recovery-dependent utility specifications. The main result is an exact characterization of universal redundancy for a subfamily of Cressie–Read power-divergence OCEs on atomless probability spaces. In this setting, in contrast to RecAV@R, where redundancy is governed by pairwise dominance, a recovery regime can be implied jointly by several other regimes. Universal redundancy is therefore a genuinely geometric property of the recovery levels and divergence-scale parameters, providing a direct screening criterion for recovery specifications.

As a second contribution, we develop a piecewise-linear approximation framework for nonlinear OCEs that turns finite-scenario portfolio problems into linear programs while retaining certified error bounds. Uniform utility approximation yields convergence of optimal values and stability of optimal solutions, and the numerical study illustrates both the portfolio effects of recovery requirements and the accuracy of the resulting approximations.

\section*{Acknowledgments}
ChatGPT (OpenAI; GPT-5.6 Sol) was used during the preparation of this manuscript to assist with literature research, the development and simplification of mathematical arguments and proofs, and the generation and revision of code for the numerical study and graphical presentation. All mathematical content and computational results were independently verified by the authors, and all cited works were consulted and checked by the authors. The authors assume responsibility for all content.

The authors thank Céline Duong for her work on Recovery OCE risk measures in her Master’s thesis, supervised by the authors, and for valuable discussions. We also thank Quentin Castro Niemann for helpful discussions on OCE approximations during his Master’s thesis.

\printbibliography

\appendix
\section{Additional proofs}
\label{app:proofs}

\subsection{Proof of \Cref{lemma:oce-sup-attained}}
\label{app:oce-sup-attained}

\begin{proof}
Set
\[
    v(x):=u(x)-x.
\]
Since \(1\in\partial u(0)\) and \(u(0)=0\), we have $v(x)\leq v(0)=0$ for all $x\in\mathbb R$. Moreover, \(v\) is concave. Hence, since it attains its maximum at zero, \(v\) is non-decreasing on \((-\infty,0]\) and non-increasing on \([0,\infty)\). Writing \(F(m):=m+\mathbb E[u(X-m)]\), we obtain
\(F(m)=\mathbb E[X]+\mathbb E[v(X-m)]\).

If \(m<x_{\min}\), then
\[
    X-m\geq X-x_{\min}\geq0
    \qquad\text{a.s.},
\]
and the monotonicity of \(v\) on \([0,\infty)\) yields $F(m)\leq F(x_{\min})$. Similarly, if \(m>x_{\max}\), then
\[
    X-m\leq X-x_{\max}\leq0
    \qquad\text{a.s.},
\]
and the monotonicity of \(v\) on \((-\infty,0]\) gives $F(m)\leq F(x_{\max})$. Consequently,
\[
    \sup_{m\in\mathbb R}F(m)
    =
    \sup_{m\in[x_{\min},x_{\max}]}F(m).
\]

It remains to show that the latter supremum is attained. For every
\(m\in[x_{\min},x_{\max}]\),
\[
    X-m\in
    [x_{\min}-x_{\max},\,x_{\max}-x_{\min}]
    \qquad\text{a.s.}
\]
By assumption, \(u\) is finite on this compact interval. Since u is closed and concave, it is continuous and hence, by the Heine–Cantor theorem, uniformly continuous. For \(m,m'\in[x_{\min},x_{\max}]\),
\[
    |F(m)-F(m')|
    \leq |m-m'|
    + \mathbb E\!\left[
        |u(X-m)-u(X-m')|
    \right].
\]
By uniform continuity of \(u\), the expectation on the right-hand side
converges to zero as \(m'\to m\), while \(|m-m'|\to0\). Thus, \(F\) is
continuous on \([x_{\min},x_{\max}]\). Since this interval is compact, \(F\)
attains its maximum there. Therefore,
\[
    S_u(X)
    =
    \max_{m\in[x_{\min},x_{\max}]}
    \left\{
        m+\mathbb E[u(X-m)]
    \right\}.
\]
\end{proof}

\subsection{Proof of \Cref{prop:properties-recrho}}
\label{app:proof-properties-recrho}
\begin{proof}
Cash invariance, monotonicity, and joint convexity follow directly from the corresponding properties of the underlying OCE risk measures and \Cref{def:recoce}. Now let \(Y\in L^\infty_+\). For every
\(\lambda\in[0,1]\),
\[
    \rho_{u_{\gamma(\lambda)}}((1-\lambda)Y)\leq 0,
\]
with equality for \(\lambda=1\), which yields
\(\operatorname{Rec}\rho_{\mathfrak u,\gamma}(0,Y)=0\). Moreover,
\[
    \rho_{u_{\gamma(\lambda)}}\bigl(X+(1-\lambda)Y\bigr)
    \leq \rho_{u_{\gamma(\lambda)}}(X)
    \leq \rho_{u_{\gamma(0)}}(X),
\]
where the second inequality follows from the monotonicity of \(\gamma\)
and the ordering of the OCE risk measures. Taking the supremum over
\(\lambda\in[0,1]\) proves the upper bound.
\end{proof}

\subsection{Proof of \Cref{cor:recavar-redundancy}}
\label{app:recavar-redundancy}

\begin{proof}
Write $R_{k,Y}(X) := \operatorname{AV@R}_{\alpha_k} \bigl(X+(1-r_k)Y\bigr)$. Suppose first that, for some \(k\neq j\),
\[
    r_k\geq r_j,
    \qquad
    \alpha_k\leq\alpha_j.
\]
For every \(X\in L^\infty\) and \(Y\in L^\infty_+\),
\[
    X+(1-r_k)Y
    \leq
    X+(1-r_j)Y.
\]

Moreover, \(\alpha_k\leq\alpha_j\) implies
\(\operatorname{AV@R}_{\alpha_k}(W)
\geq\operatorname{AV@R}_{\alpha_j}(W)\) for every \(W\in L^\infty\).
Together with the preceding inequality, this yields
\(R_{k,Y}(X)\geq R_{j,Y}(X)\), and hence regime \(j\) is universally
redundant.

Conversely, suppose that no \(k\neq j\) satisfies both
\(r_k\geq r_j\) and \(\alpha_k\leq\alpha_j\). Let
\(A\in\mathcal F\) satisfy \(\mathbb P(A)=\alpha_j\). Using the dual representation of AV@R, for \(a>0\),
\begin{equation}\label{eq:proofA4}
\operatorname{AV@R}_\alpha(-a\mathbf 1_A)
=
a\sup_{\substack{Z\geq0,\ \mathbb E[Z]=1\\ Z\leq1/\alpha}}
\mathbb E[Z\mathbf 1_A]
=
a\min\left\{1,\frac{\mathbb P(A)}{\alpha}\right\}.
\end{equation}
Set \(s_k:=1-r_k\) and fix \(y>0\). For \(x>y\max_k s_k\), define \(X=-x\mathbf 1_A\) and \(Y=y\mathbf 1_A\). Then, by \eqref{eq:proofA4},
\[
    R_{k,Y}(X)
    =
    (x-s_ky)
    \min\left\{1,\frac{\alpha_j}{\alpha_k}\right\},
    \qquad
    R_{j,Y}(X)=x-s_jy.
\]
If \(\alpha_k\leq\alpha_j\), the assumption implies \(r_k<r_j\), hence
\(s_k>s_j\), and therefore
\[
    R_{k,Y}(X)=x-s_ky<x-s_jy=R_{j,Y}(X).
\]
If \(\alpha_k>\alpha_j\), set
\(q_k:=\alpha_j/\alpha_k\in(0,1)\). Then
\[
\begin{aligned}
    R_{j,Y}(X)-R_{k,Y}(X)
    &=
    (1-q_k)x-y(s_j-q_ks_k),
\end{aligned}
\]
which is strictly positive for all sufficiently large \(x\). Since there
are only finitely many regimes, \(x\) can be chosen so that this holds
simultaneously for every \(k\neq j\). Consequently,
\[
    R_{j,Y}(X)
    >
    \max_{k\neq j}R_{k,Y}(X),
\]
so regime \(j\) is not universally redundant. This proves the converse by
contraposition.

If the probability space is atomless, an event \(A\) with
\(\mathbb P(A)=\alpha_j\) exists for every \(\alpha_j\in(0,1)\), which
proves the final assertion.
\end{proof}

\subsection{Proof of \Cref{theorem:powerUniversalRedundancy}}
\label{app:powerUniversalRedundancy}
\begin{proof}

The equivalence of \textup{(ii)} and \textup{(iii)} follows directly from
the definition \(q_k=(\theta_k,1-r_k)\). Indeed,
\(q_j\in\operatorname{conv}\{q_k:k\neq j\}+\mathbb R_+^2\) if and only if
there exist weights \(\xi_k\geq0\), \(k\neq j\), with
\(\sum_{k\neq j}\xi_k=1\) such that
\[
\theta_j\geq\sum_{k\neq j}\xi_k\theta_k,
\qquad
1-r_j\geq\sum_{k\neq j}\xi_k(1-r_k),
\]
where the second inequality is equivalent to
\(r_j\leq\sum_{k\neq j}\xi_k r_k\).

The implication \textup{(iii)}$\Rightarrow$\textup{(i)} follows from \Cref{proposition:structuralScreening}. Since \(\phi_{\theta_k,p}=\theta_k\phi_p\) and \(\phi_p\geq0\), condition \textup{(iii)} implies
\[
\phi_{\theta_j,p}(z)
\geq
\sum_{k\neq j}\xi_k\phi_{\theta_k,p}(z),
\qquad z\geq0.
\]
Together with $r_j \leq \sum_{k\neq j}\xi_k r_k$, \Cref{proposition:structuralScreening} yields \textup{(i)}.

It remains to prove
\textup{(i)}$\Rightarrow$\textup{(ii)}. We argue by contraposition. To this end, set \(s_k:=1-r_k\) and \(C_{-j}:=\operatorname{conv}\{q_k:k\neq j\}+\mathbb R_+^2\). Suppose that \(q_j\notin C_{-j}\). Since \(C_{-j}\) is a nonempty closed convex subset of \(\mathbb R^2\), strict separation yields a nonzero vector \((h,a)\in\mathbb R^2\) such that
\[
h\theta_j+as_j
<
\inf_{q\in C_{-j}}\langle(h,a),q\rangle.
\]
Because \(C_{-j}\) is upper closed with respect to the componentwise order,
both components of the separating vector must be nonnegative; otherwise
the infimum on the right-hand side would equal \(-\infty\). Hence
\((h,a)\in\mathbb R_+^2\setminus\{(0,0)\}\). Moreover,
\[
\inf_{q\in C_{-j}}\langle(h,a),q\rangle
=
\min_{k\neq j}\{h\theta_k+as_k\},
\]
and therefore
\begin{equation}
h\theta_j+as_j
<
h\theta_k+as_k,
\qquad k\neq j.
\tag{*}
\end{equation}

We construct an asset--liability position for which regime \(j\) is strictly
active, distinguishing the cases \(h=0\) and \(h>0\).

Suppose first that \(h=0\). Then \(a>0\), and \textup{(*)} implies
\(s_j<s_k\) for every \(k\neq j\). Choose any \(c>0\) and set \(X=0\)
and \(Y=c\). By normalization and cash invariance of the OCE risk measures,
\[
R_{k,Y}(X)
=
\rho_{u_{\theta_k,p}}(s_kc)
=
-s_kc.
\]
Consequently,
\[
R_{j,Y}(X)
=
-s_jc
>
-s_kc
=
R_{k,Y}(X),
\qquad k\neq j.
\]
Thus regime \(j\) is uniquely active and therefore not universally
redundant.

Suppose now that $h>0$. We first consider $p>1$. For sufficiently large
$t$, define
\[
    \varepsilon_t
    :=
    \bigl(p(p-1)ht\bigr)^{-1/(p-1)}
    \in(0,1).
\]
By atomlessness, there exists an event $A_t\in\mathcal F$ satisfying $\mathbb P(A_t)=\varepsilon_t$.

Choose $x>0$ sufficiently large that
\begin{equation}
    x-as_k>ph\theta_k,
    \qquad k=1,\ldots,l+1,
    \tag{**}
\end{equation}
and define
\[
    X_t:=-tx\,\mathbf 1_{A_t},
    \qquad
    Y_t:=ta\,\mathbf 1_{A_t}.
\]
Then \(X_t+s_kY_t=-t(x-as_k)\mathbf 1_{A_t}\). For brevity, write
\(d_k:=x-as_k.\)

Using the dual representation of the power-divergence OCE gives
\[
    R_{k,Y_t}(X_t)
    =
    \sup_{\substack{Z\geq0\\ \mathbb E[Z]=1}}
    \left\{
        t d_k\,\mathbb E[Z\mathbf 1_{A_t}]
        -
        \theta_k\mathbb E[\phi_p(Z)]
    \right\}.
\]
For a prescribed value
\(q:=\mathbb E[Z\mathbf 1_{A_t}]\in[0,1]\), we have
\[
\mathbb E[Z\mid A_t]
=
\frac{q}{\varepsilon_t},
\qquad
\mathbb E[Z\mid A_t^c]
=
\frac{1-q}{1-\varepsilon_t},
\]
since \(\mathbb E[Z]=1\). Applying Jensen's inequality conditionally on
$A_t$ and $A_t^c$ therefore gives
\[
    \mathbb E[\phi_p(Z)]
    \geq
    \varepsilon_t
    \phi_p\left(\frac{q}{\varepsilon_t}\right)
    +
    (1-\varepsilon_t)
    \phi_p\left(\frac{1-q}{1-\varepsilon_t}\right).
\]
This lower bound is attained by
\[
    Z_{t,q}
    =
    \frac{q}{\varepsilon_t}\mathbf 1_{A_t}
    +
    \frac{1-q}{1-\varepsilon_t}\mathbf 1_{A_t^c},
\]
which satisfies $Z_{t,q}\geq0$, $\mathbb E[Z_{t,q}]=1$, and
$\mathbb E[Z_{t,q}\mathbf 1_{A_t}]=q$. Hence, for each fixed $q$,
$Z_{t,q}$ minimizes the divergence penalty among all admissible densities
with $\mathbb E[Z\mathbf 1_{A_t}]=q$. Consequently,
\begin{equation}
    R_{k,Y_t}(X_t)
    =
    \sup_{q\in[0,1]}
    \left\{
        t d_kq-\theta_kD_{p,t}(q)
    \right\},
    \tag{***}
\end{equation}
where
\[
\begin{aligned}
    D_{p,t}(q)
    &:=
    \mathbb E[\phi_p(Z_{t,q})]
    \\
    &=
    \frac{
        q^p\varepsilon_t^{\,1-p}
        +(1-q)^p(1-\varepsilon_t)^{\,1-p}
        -1
    }{p(p-1)}.
\end{aligned}
\]
By the definition of $\varepsilon_t$,
\[
    \frac{\varepsilon_t^{\,1-p}}{p(p-1)t}=h.
\]
Moreover,
\[
    \sup_{q\in[0,1]}
    \left|
        \frac{
            (1-q)^p(1-\varepsilon_t)^{\,1-p}-1
        }{p(p-1)t}
    \right|
    \longrightarrow0.
\]
It follows that \(D_{p,t}(q)/t\to hq^p\) uniformly for
\(q\in[0,1]\). Dividing \textup{(***)} by $t$ and using
uniform convergence therefore yields
\begin{equation}
    \lim_{t\to\infty}
    \frac{1}{t}R_{k,Y_t}(X_t)
    =
    \sup_{q\in[0,1]}
    \{d_kq-h\theta_kq^p\}.
    \tag{****}
\end{equation}
By \textup{(**)}, \(d_k-ph\theta_k>0\). Thus, for every $q\in[0,1]$,
\[
    \frac{d}{dq}
    \bigl(d_kq-h\theta_kq^p\bigr)
    =
    d_k-ph\theta_kq^{p-1}
    >0.
\]
The supremum in \textup{(****)} is therefore attained at $q=1$, and
\begin{equation}
    \lim_{t\to\infty}
    \frac{1}{t}R_{k,Y_t}(X_t)
    =
    x-as_k-h\theta_k.
    \tag{*****}
\end{equation}
By \textup{(*)},
\[
    x-as_j-h\theta_j
    >
    x-as_k-h\theta_k,
    \qquad k\neq j.
\]
Since there are only finitely many regimes, \textup{(*****)} implies that,
for all sufficiently large $t$,
\[
    R_{j,Y_t}(X_t)
    >
    \max_{k\neq j}R_{k,Y_t}(X_t).
\]
Hence regime $j$ is uniquely active and therefore not universally
redundant.

It remains to treat $p=1$ when $h>0$. In this case,
\(\phi_{\theta_k,1}(z)=\theta_k(z\log z-z+1),\)
so that the associated OCE risk measure is entropic:
\[
    \rho_{\theta_k,1}(W)
    =
    \theta_k
    \log\mathbb E
    \left[
        \exp\left(-\frac{W}{\theta_k}\right)
    \right].
\]
By atomlessness, for every $t>0$ there exists an event $A_t\in\mathcal F$
such that \(\mathbb P(A_t)=e^{-th}. \)
Choose \(x\) such that
\(x>as_k+h\theta_k\) for all \(k=1,\ldots,l+1\), and define
\[
X_t:=-tx\,\mathbf 1_{A_t},
\qquad
Y_t:=ta\,\mathbf 1_{A_t}.
\]
Then
\[
\begin{aligned}
    R_{k,Y_t}(X_t)
    &=
    \theta_k
    \log
    \left(
        1-e^{-th}
        +
        \exp\left(
            \frac{t[x-as_k-h\theta_k]}{\theta_k}
        \right)
    \right),
\end{aligned}
\]
and hence
\[
    \lim_{t\to\infty}
    \frac{1}{t}R_{k,Y_t}(X_t)
    =
    x-as_k-h\theta_k.
\]
Again, \textup{(*)} implies
\(x-as_j-h\theta_j>x-as_k-h\theta_k\) for every \(k\neq j\).
Thus, for all sufficiently large \(t\),
\[
R_{j,Y_t}(X_t)
>
\max_{k\neq j}R_{k,Y_t}(X_t),
\]
so regime \(j\) is not universally redundant.

We have therefore shown in all cases that
$q_j\notin C_{-j}$ implies that regime $j$ is not universally redundant.
This proves \textup{(i)}$\Rightarrow$\textup{(ii)} by contraposition and
completes the proof.
\end{proof}

\subsection{Proof of \Cref{theorem:portfolio-stability}}
\label{app:portfolio-stability}
\begin{proof}
For \(\mathbf w\in\Delta_K(\mu)\) and \(j=1,\ldots,l+1\), let
\(I_j(\mathbf w)\) denote the position-specific residual interval associated
with \(\mathbf Z_j(\mathbf w)\). By construction,
\(I_j(\mathbf w)\subseteq I_j^{\mathrm{port}}\). Hence,
\Cref{lemma:recoce-error-propagation} yields
\[
    |F_n(\mathbf w)-F(\mathbf w)|\leq\varepsilon_n,
    \qquad \mathbf w\in\Delta_K(\mu).
\]
Taking the supremum over \(\Delta_K(\mu)\) proves~\textup{(i)}.

Every finite-valued monetary risk measure on \(L^\infty\) is
\(1\)-Lipschitz with respect to the \(L^\infty\)-norm. Since
\(\mathbf A(\mathbf w)\) depends continuously on \(\mathbf w\), the
regime-wise OCEs, and hence their finite maxima \(F\) and \(F_n\), are
continuous on the compact set \(\Delta_K(\mu)\). They therefore attain
their minima. Part~\textup{(i)} implies
\(V_n\leq V+\varepsilon_n\) and \(V\leq V_n+\varepsilon_n\), so that
\(|V_n-V|\leq\varepsilon_n\) and \(V_n\to V\). This proves~\textup{(ii)}.

Let
\(\mathbf w_n\in
\operatorname*{arg\,min}_{\mathbf w\in\Delta_K(\mu)}F_n(\mathbf w)\),
and let \(\overline{\mathbf w}\) be any accumulation point of
\((\mathbf w_n)\). By compactness, such accumulation points exist, and
there is a subsequence
\(\mathbf w_{n_k}\to\overline{\mathbf w}\). By~\textup{(i)} and optimality,
\[
    F(\mathbf w_{n_k})
    \leq V_{n_k}+\varepsilon_{n_k}.
\]
Using continuity of \(F\), part~\textup{(ii)}, and
\(\varepsilon_{n_k}\to0\) gives
\(F(\overline{\mathbf w})\leq V\). Since \(V\) is the minimum of \(F\),
the reverse inequality holds, and hence
\(F(\overline{\mathbf w})=V\). Thus every accumulation point of
\((\mathbf w_n)\) is an exact minimizer. If the exact minimizer
\(\mathbf w^\star\) is unique, then \(\mathbf w^\star\) is the only
accumulation point of \((\mathbf w_n)\); compactness therefore implies
\(\mathbf w_n\to\mathbf w^\star\). This proves~\textup{(iii)}.
\end{proof}

\section{An Example of Strict Convexification}
\label{app:strict-convexification}

Consider a two-state probability space with
$\mathbb P=(1/2,1/2)$ and let $Y=(0,1)$. For $q\in[0,1]$, define
\[
    Q_q=(q,1-q),
    \qquad
    Z_q:=\frac{dQ_q}{d\mathbb P}=(2q,2(1-q)).
\]
Then
\[
    \mathbb E[Z_qY]
    =
    \mathbb E_{Q_q}[Y]
    =
    1-q.
\]

Consider two recovery regimes with
\[
    0<\theta_1<\theta_2,
    \qquad
    r_1<r_2,
\]
and let their OCEs be generated by the entropic divergence functions
\[
    \phi_{\alpha_j}(z)=\theta_j(z\log z-z+1),
    \qquad j=1,2.
\]
By the general form of the regime-wise minimal penalties, using the convention $0\log 0:=0$,
\[
\begin{aligned}
g_{j,Y}(Z_q)
&=
\mathbb E[\phi_{\alpha_j}(Z_q)]
+
(1-r_j)\mathbb E[Z_qY] \\
&=
\theta_j
\bigl[
q\log(2q)+(1-q)\log(2(1-q))
\bigr]
+
(1-r_j)(1-q),
\qquad j=1,2.
\end{aligned}
\]
The two terms in these penalties induce opposite orderings across the
regimes. Indeed, $\theta_1<\theta_2$ makes the divergence term smaller
under the first regime, whereas $r_1<r_2$ makes its recovery term larger.
This trade-off leads to a crossing of the two penalty functions. To see
this, define
\[
\begin{aligned}
h(q):=g_{1,Y}(Z_q)-g_{2,Y}(Z_q)
&=
-(\theta_2-\theta_1)
\bigl[
q\log(2q)+(1-q)\log(2(1-q))
\bigr] \\
&\quad +(r_2-r_1)(1-q).
\end{aligned}
\]
Since
\[
h(1/2)=\frac{r_2-r_1}{2}>0,
\qquad
h(1)=-(\theta_2-\theta_1)\log 2<0,
\]
the two penalties cross at some $q^*\in(1/2,1)$. The crossing is unique,
since
\[
h'(q)
=
-(\theta_2-\theta_1)\log\frac{q}{1-q}
-(r_2-r_1)
<0,
\qquad q\in(1/2,1).
\]
Consequently,
\[
\beta_Y(Z_q)
=
\min\{g_{1,Y}(Z_q),g_{2,Y}(Z_q)\}
\]
switches from $g_{2,Y}(Z_q)$ to $g_{1,Y}(Z_q)$ at $q^*$. Moreover,
\[
\frac{d}{dq}g_{1,Y}(Z_q)\Big|_{q=q^*}
<
\frac{d}{dq}g_{2,Y}(Z_q)\Big|_{q=q^*},
\]
because $h'(q^*)<0$. Thus, $q\mapsto\beta_Y(Z_q)$ has a downward jump in slope at $q^*$ and is therefore not convex. Since $q\mapsto Z_q$ is affine, $\beta_Y$ itself cannot be convex. Consequently, the convex envelope of $\beta_Y$ is strictly smaller than $\beta_Y$ at some point. Hence, the convexification in the minimal penalty representation can be strict even for two entropic OCE regimes.

\end{document}